\documentclass[a4paper]{article}

\usepackage{graphicx} % Required for inserting images
\usepackage{hyperref}
\usepackage{amssymb,amsfonts,amsthm,amsmath,mathtools,cleveref,comment,mathrsfs}
\usepackage{pgf}
\usepackage[margin=1in]{geometry}
\usepackage{xcolor}
\usepackage{paralist} %compactitem
\usepackage{tikz-cd}
\usepackage{todonotes}
\usepackage[backend=biber,style=alphabetic,sorting=nyt,maxbibnames=99,maxalphanames=5,doi=false,isbn=false,url=false]{biblatex}
\usepackage[normalem]{ulem}
\usepackage{float}
\usepackage{thm-restate}
\usepackage{epstopdf}
\usepackage{listings}
\usepackage{soul,cancel}
\usepackage{enumitem}

\usetikzlibrary{arrows}

\newtheorem{theorem}{Theorem}[section]
\newtheorem{lemma}[theorem]{Lemma}
\newtheorem{proposition}[theorem]{Proposition}
\newtheorem{corollary}[theorem]{Corollary}
\theoremstyle{definition}
\newtheorem{definition}[theorem]{Definition}
\theoremstyle{remark}
\newtheorem{remark}[theorem]{Remark}

\newtheorem{example}[theorem]{Example}

\newtheorem{claim}[theorem]{Claim}

\numberwithin{table}{section}

\newcommand{\defn}[1]{\ul{#1}}
\newcommand\cD{{\mathcal D}}

\newcommand\cG{{\mathcal G}}
\newcommand\cH{{\mathcal H}}
\newcommand\cI{{\mathcal I}}

\newcommand\cM{{\mathcal M}}
\newcommand\cP{{\mathcal P}}
\newcommand\cQ{{\mathcal Q}}

\newcommand\BB{{\mathbb B}}
\newcommand\CC{{\mathbb C}}

\newcommand\KK{{\mathbb K}}

\newcommand\PP{{\mathbb P}}

\newcommand\RR{{\mathbb R}}

\newcommand\ZZ{{\mathbb Z}}

\newcommand\bbG{{\mathbb G}}
\newcommand\bbP{{\mathbb P}}

\renewcommand{\st}{\colon} % st == such that (in case we want to switch to | in the future or add/remove space around)
\newcommand{\set}[2]{\left\{#1 \st #2\right\}} % Short hand for { something \st something} with appropriate brackets
\newcommand{\symDiff}{\mathbin\triangle} % Symmetric Difference - Considering we have Delta matroids, maybe a different symbol might be more useful?
\newcommand{\sd}[2]{{#1 \symDiff \{#2\}}}

\renewcommand{\emptyset}{\varnothing}
\renewcommand{\bar}{\overline}
\DeclarePairedDelimiter\order{\lvert}{\rvert}%

\newcommand{\ceil}[1]{\left\lceil #1 \right\rceil}

\DeclareMathOperator{\Rank}{rank}
\DeclareMathOperator{\proj}{proj}
\DeclareMathOperator{\conv}{conv}

\DeclareMathOperator{\OGr}{OGr} % Lie

\newcommand{\Sp}{\mathcal{S}}

\newcommand{\polyOf}[1]{P(#1)}

\newcommand{\cube}[2]{\square_{#1, #2}} % Ben: Changed to \square as Kieran had `accidentally' used a different definition using Q
\newcommand{\cubeVertex}[1]{\varepsilon_{#1}}

\newcommand{\evencube}[2]{D^0_{#1, #2}}
\newcommand{\oddcube}[2]{D^1_{#1, #2}}

\newcommand{\hypersimplex}[2]{\cH_{#1, #2}}

\newcommand{\matroid}{M}

\DeclareMathOperator{\trop}{trop}

\begin{document}
\title{Characterisations of strong $\Delta$-matroids}
\author{Kieran Calvert\thanks{Lancaster University, E-mail: \texttt{kieran.calvert@lancaster.ac.uk}} \and 
Aram Dermenjian\thanks{Universidad de Sevilla, E-mail: \texttt{aram.dermenjian.math@gmail.com}} \and 
Alex Fink\thanks{Queen Mary University of London, E-mail: \texttt{a.fink@qmul.ac.uk}} \and 
Ben Smith\thanks{Lancaster University, E-mail: \texttt{b.smith9@lancaster.ac.uk}}}
\date{}

\maketitle

\abstract{
We study characterisations of strong $\Delta$-matroids, compiling a list of five equivalent descriptions.
We show a variant of Wenzel's exchange property and the hyperplane exchange property of Borovik--Gelfand--White are equivalent.
We also introduce two novel characterisations in terms of `peerless' and `isolated' antipodes within the system of feasible sets, banning certain configurations of antipodes either globally or locally.
As a corollary, we obtain new `local' exchange axioms for matroids and $\Delta$-matroids.
We give algebraic motivation for these new characterisations by introducing the peerless antipode equations, tropical equations that govern whether a $\Delta$-matroid has no peerless antipodes.
We show that these arise as the tropicalisation of a specific basis of quadratics cutting out the orthogonal Grassmannian.}

\section{Introduction}

The observation goes back to Brualdi \cite{Brualdi1969}
that the ordinary basis exchange axiom for matroids is equivalent to the symmetric exchange axiom.
$\Delta$-matroids are also defined by a basis exchange axiom,
but not all $\Delta$-matroids satisfy a counterpart of symmetric exchange: the smallest example that doesn't has ground set of size~$3$. 
Those $\Delta$-matroids that do are called \emph{strong}.

In fact, multiple different-looking symmetric exchange properties for $\Delta$-matroids appear in the literature.
Wenzel in \cite{Wenzel1993} formulated $\Delta$-matroid symmetric exchange in terms of set systems.
Later, Borovik, Gelfand and White in their book \cite{BorovikGelfandWhite2003} interpreted both matroids and $\Delta$-matroids (which they referred to as Lagrangian matroids)
as cases of Coxeter matroids, associated with different choices of the ``type'' data of an orbit polytope for a reflection group.
In this context they gave an axiom of the symmetric exchange flavour using hyperplanes,
stated uniformly over all classes of Coxeter matroids.
See \Cref{ss:strong} for further details.
For $\Delta$-matroids, however, Borovik, Gelfand and White also made use of one of Wenzel's properties without comparing it to the hyperplane property.
We carried over this omission in our previous paper \cite{CDFS1}.
In the present paper we prove their equivalence, 
as the assertion \eqref{it:1-strong}$\Leftrightarrow$\eqref{it:2-weak-wenzel} of our first main theorem, \Cref{thm:A} below.

The novel equivalences in \Cref{thm:A} are those involving antipodes. These have the following combinatorial motivation.
Some of the axiom systems for matroids can be formulated in terms of either arbitrary sets or only sets that are close together in (say) Hamming distance. We think of these as ``global'' and ``local'' formulations.
For example, when describing submodularity of the rank function, it is equivalent to ask that
\[
    \Rank(A)+\Rank(B)\ge\Rank(A\cup B)+\Rank(A\cap B)
\]
for all pairs of subsets $A,B$ of the ground set (the global choice), 
or only for those pairs with $\order{A\symDiff B}=2$ (the local choice).

Consider what a local formulation of matroid basis exchange would be.
%our focus in this paper is symmetric exchange, but for the moment think of ordinary exchange.
If the exchange axiom has nontrivial content for two bases $B$ and $B'$, 
then the Hamming distance $\order{B'\symDiff B}$ is at least $4$.
But there are set systems that satisfy (symmetric) basis exchange ``locally'', 
for all pairs $B,B'$ of bases with $\order{B'\symDiff B}=4$, but are not matroids.
As has been observed many times, e.g.\ by Baker and Lorscheid in \cite[Example 6.25]{BakerLorscheid}, the set system $\{\{1,2,3\},\{4,5,6\}\}$
vacuously satisfies this local condition, containing no pairs $(B,B')$ of distinct sets with $\order{B'\symDiff B}=4$,
but is not the set of bases of a matroid.

This theme has been important in the literature on matroids with  coefficients. 
For valuated matroids, the \emph{3-term tropical Pl\"ucker relations} are a local valuated exchange axiom for all pairs $B,B'$ of bases with $\order{B'\symDiff B}=4$.
One definition for a valuated matroid is that it satisfies this local valuated exchange axiom \emph{and} has matroidal support \cite[Theorem 5.2.25]{Murota:2010};
the inclusion of the second clause covers for the failure of local basis exchanges to imply all basis exchanges.
More generally, the theory of matroids over partial hyperstructures defines a \emph{weak matroid} to be one which satisfies a local exchange axiom \emph{and} has matroidal support \cite{BakerBowler}, for the same reasons as for valuated matroids (see \Cref{rem:strong-vs-weak} for further details).
%The definition often given for a ``tropical Pl\"ucker vector'' \alex{Recast to be about Murota, Matrices and Matroids for Systems Analysis, 5.2.25? doesn't use this language but does use ``local''!}
%is a vector which satisfies the 3-term Pl\"ucker relations \emph{and} has matroidal support;
%the inclusion of the second clause covers for the failure of local basis exchanges to imply all basis exchanges.

In this paper we give characterisations, for both matroids and strong $\Delta$-matroids,
in the spirit of how little we must add to local symmetric exchange to recover the full symmetric exchange axiom.
As mentioned, both matroids and $\Delta$-matroids are cases of Coxeter matroids.
Like matroids do, Coxeter matroids have various cryptomorphic definitions. For our purposes, we will use the polyhedral definition and we switch to polyhedral language now.
Matroids are certain subpolytopes of hypersimplices whose vertices are in bijection with the bases;
likewise, $\Delta$-matroids are certain subpolytopes of cubes whose vertices are in bijection with the feasible sets.
Hypersimplices and cubes share the property that any two vertices are antipodal in some face $F$ that is centrally symmetric up to translation.
Thus, given a subpolytope $P$ of a hypersimplex or cube,
we will call any ordered pair $(B,B')$ of vertices of~$P$ an \emph{antipode}, 
thinking of this pair of vertices within the face $F$.

If $B$ and $B'$ are two vertices of~$P$ representing bases or feasible sets, 
then finding a strong exchange between $B$ and~$B'$ translates to
finding another nearby pair $(C, C')$ of vertices such that $(B, B')$ and $(C, C')$ are antipodes of the same face $F$ of the ambient polytope.
We will say that the antipode $(B,B')$ is \emph{peerless} if $P$ contains no other antipodes in~$F$,
and that $(B,B')$ is \emph{isolated} if $P$ contains no other vertices in~$F$ at all.
Then \Cref{thm:A}(\ref{it:5-antipodes}) says that $P$ is a strong $\Delta$-matroid
if and only if it contains no peerless antipodes in faces of the smallest nontrivial dimensions,
and no isolated antipodes in any (higher-dimensional) face.
The counterpart for matroids, which we give in \Cref{ssec:matroids-antipodes}, is folklore, related to the argument sketched in the end of \cite[Section 1.9]{BorovikGelfandWhite2003}.
In our characterisations, the peerless condition is the ``local'' case of strong exchange.
The isolated condition is the weakest possible ``global'' statement we can use to supplement the local statement to an equivalence,
in the sense that up to symmetry, it excludes only one minimal counterexample subpolytope in each dimension of face, namely the long diagonal.
Without further ado, here is the full \Cref{thm:A}.

\renewcommand{\thetheorem}{A}
\begin{theorem} \label{thm:A}
    Let $M$ be a $\Delta$-matroid.
    The following are equivalent:
    \begin{enumerate}
        \item $M$ is a strong $\Delta$-matroid (i.e.\ satisfies the hyperplane exchange property).\label{it:1-strong}
        \item $M$ satisfies the weak Wenzel exchange property.\label{it:2-weak-wenzel}
        \item $\bar{M}$ is an even $\Delta$-matroid.\label{it:3-even-Delta}
        \item $M$ has no peerless antipode in any $k$-cube for $k \geq 3$.\label{it:4-cubes}
        \item $M$ has no peerless antipode in any $3$-cube or $4$-cube and has no isolated antipode in any $k$-cube for $k \geq 5$.\label{it:5-antipodes}
    \end{enumerate}
\end{theorem}

By a result by Murota \cite[Theorem 3.1]{Murota2021}, it was previously known that \eqref{it:2-weak-wenzel} and \eqref{it:3-even-Delta} are equivalent.
In the following we show that \eqref{it:1-strong}, \eqref{it:2-weak-wenzel} and \eqref{it:3-even-Delta} are equivalent in \Cref{thm:hyperplane-equals-wenzel}.
We show \eqref{it:1-strong} and \eqref{it:4-cubes} are equivalent in \Cref{thm:strong+iff+no+peerless} and the equivalency with \eqref{it:5-antipodes} in \Cref{prop:type+2+B}.

Our introduction of peerless antipodes also has algebraic motivations.
Analogous to the relationship between matroids and the Grassmannian, $\Delta$-matroids on $[n]$ are related to the orthogonal Grassmannian $\OGr(n,2n+1)$, also known as the maximal isotropic Grassmannian.
We shed light on this relationship in a previous paper \cite{CDFS1}, demonstrating that there exists a defining set of equations for $\OGr(n,2n+1)$ whose tropicalisation characterises strong $\Delta$-matroids.
These tropical equations are called the \emph{strong exchange equations}, and combinatorially characterise when a set system satisfied the hyperplane exchange property.
In this article, we introduce the \emph{peerless antipode equations} $\cG_n$ (see \Cref{defn:peerless+antipode+equations}), tropical equations that characterise when a set system has no peerless antipodes.
These equations do not live in a vacuum: they arise as the tropicalisation of a very special family of equations cutting out $\OGr(n,2n+1)$.

\renewcommand{\thetheorem}{B} 
\begin{theorem}\label{thm:B}
    Let $\OGr(n,2n+1) \subset \PP^{2^n-1}$ denote the orthogonal Grassmannian with defining ideal $\cI_n$.
    There exists a generating set $\cP_n$ of $\cI_n$ such that:
    \begin{itemize}
        \item $\cP_n$ forms a basis for the degree two graded piece of $\cI_n$,
        \item the tropicalisation of $\cP_n$ is $\cG_n$.
    \end{itemize}
\end{theorem}
\renewcommand{\thetheorem}{\arabic{section}.\arabic{theorem}}

In particular, when coupled with \Cref{thm:A} it follows that a set system on $[n]$ is a strong $\Delta$-matroid if and only if it satisfies the tropicalisation of $\cP_n$.

Despite \Cref{thm:B} being an algebraic statement, we take an entirely combinatorial viewpoint on its proof.
The first part of the statement was already shown in \cite{CDFS1}, which we recall in \Cref{t:BspinEqs}.
The second part is proved in \Cref{prop:tropicalisation} and is entirely combinatorial.

Unlike what's true for the strong exchange equations, many equations in $\cP_n$ tropicalise to the same peerless antipode equation.
In particular, there are many lower support equations in the span of $\cP_n$ that are forgotten in the tropicalisation process. The approach in \cite{CDFS1} was to only tropicalise equations of minimal support to avoid this problem.
It is somewhat surprising then that the peerless antipode equations are sufficient to characterise strong $\Delta$-matroids (see Remark \ref{rem:surprising tropicalisation}).

\paragraph*{Acknowledgements} 
The authors thank Matt Baker and Joseph Kung for their insight into the literature,
and Matt Larson for enlightening discussions.

The second author received support as part of the research project PID2022-138719NA-I00, financed by MCIN/AEI/10.13039/501100011033/FEDER, UE.
The third author received support from the Engineering and Physical Sciences Research Council [grant number EP/X001229/1].
The fourth author received support from the Engineering and Physical Sciences Research Council [grant number EP/X036723/1].

\section{Exchange properties for $\Delta$-matroids}

\subsection{$\Delta$-matroids}
Throughout we let $M$ be a set system on $[n]$.
We say that $\matroid$ is an (ordinary) \defn{matroid} if it satisfies the following basis exchange property:
\[
    \text{for all } A, B \in \matroid \, , \, a \in A \setminus B , \text{ there exists } b \in B \setminus A\text{ such that } A \setminus \{a\} \cup \{b\} \in \matroid\, .
\]
We we say a set system is \emph{equicardinal} if all sets have the same cardinality.
It is well-known that matroids are equicardinal. 
%There are many cryptomorphic definitions, but the one we will focus on is the strong exchange property for ordinary matroids, as given in \cite[Section 1.5.1]{BorovikGelfandWhite2003}.

% A set system $\matroid$ on $[n]$ satisfies the \defn{strong exchange property for (ordinary) matroids} if
% \[
%     \text{for all distinct }A, B \in \matroid \text{ there exist } a \in A \setminus B,\,b\in B\setminus A \text{ such that } A \symDiff\left\{ a, b \right\}, B\symDiff\left\{ a, b \right\} \in \matroid \, .
% \]
% It is shown in \cite[Theorem 1.5.2]{BorovikGelfandWhite2003} that a set system $\matroid$ satisfies the strong exchange property if and only if it is a matroid.

Matroids are a special case of a more general family of set systems called $\Delta$-matroids, defined via an exchange axiom using symmetric difference.

\begin{definition}[{\cite[Definition 6]{Bouchet1987}}]
    A \defn{$\Delta$-matroid} is a set system $\matroid$ on $[n]$ such that the following \defn{symmetric exchange property for $\Delta$-matroids} is satisfied:
    \begin{equation}\label{axiom:symm-ex}
        \text{For all } A,B \in \matroid \, , \, a \in A \symDiff B, \text{ there exists } b \in A \symDiff B \text{ such that } A \symDiff \{a, b\} \in \matroid \, .
    \end{equation}
    Note that $a$ and $b$ might be equal.
\end{definition}
It is straightforward to verify that matroids are a special case of $\Delta$-matroids.
Explicitly, $M$ is a matroid if and only if $M$ is an equicardinal $\Delta$-matroid.

Another special case of $\Delta$-matroids we are interested in are even $\Delta$-matroids.
We say that a set system $\matroid$ is \defn{even} if all $A \in \matroid$ have the same parity.
Despite the name, we emphasise that the sets in an even set system may be of odd or even cardinality.
We say that $\matroid$ is an \defn{even $\Delta$-matroid} if it is a $\Delta$-matroid and an even set system.
It follows that matroids are also a special case of even $\Delta$-matroids.

A second way to understand $\Delta$-matroids and their variants is in terms of their polytopes.
Let $\cubeVertex{1}, \dots, \cubeVertex{n}$ be an orthonormal basis for $\RR^n$.
Given $S \subseteq [n]$, we let $\cubeVertex{S} = \sum_{i\in S} \cubeVertex{i}$ be its indicator vector.
Given a set system $M$ on $[n]$, we define the polytope
\[
P(M) := \conv\set{\cubeVertex{A}}{A \in M} \subseteq \RR^n \, .
\]
The following theorem characterises $\Delta$-matroids and their variants in terms of edges of this polytope.
\begin{theorem}[{\cite{GelfandSerganova:1987,BorovikGelfandWhite2003}}]\label{thm:polytope+edges}
    Let $M$ be a set system on $[n]$ and $P(M) \subseteq \RR^n$ its associated polytope.
    \begin{enumerate}
        \item $M$ is a $\Delta$-matroid if and only if $P(M)$ has all edges parallel to some $e_i$ or $e_i \pm e_j$ where $i,j \in [n]$.
        \item $M$ is an even $\Delta$-matroid if and only if $P(M)$ has all edges parallel to some $e_i \pm e_j$ where $i,j \in [n]$.
        \item $M$ is a matroid if and only if $P(M)$ has all edges parallel to some $e_i - e_j$ where $i,j \in [n]$.
    \end{enumerate}
\end{theorem}

This theorem places $\Delta$-matroids and their variants in the framework of \defn{Coxeter matroids}, introduced by Borovik, Gelfand and White.
It suffices for us to think of our cases of Coxeter matroids as zero-one polytopes whose edges are all parallel to a root in some fixed root system $\Phi$; see \cite{BorovikGelfandWhite2003} or \cite{CDFS1} for full definitions.
In particular, \Cref{thm:polytope+edges} can be rephrased as the statement that matroids, even $\Delta$-matroids and $\Delta$-matroids are Coxeter matroids of the root systems $A_n$, $D_n$ and $B_n$ respectively.
%\benIL{Maybe need more information for $E_6/E_7$, but seemed reasonable for most of the paper}

\subsection{Strong $\Delta$-matroids}\label{ss:strong}

Borovik, Gelfand and White introduced \emph{strong Coxeter matroids} as the Coxeter matroids satisfying a ``strong exchange property''.
Unfortunately, Wenzel and authors after him had introduced different ``strong exchange properties'' for $\Delta$-matroids.
As these are a case of Coxeter matroids, the result has been terminological confusion:
every property we introduce in this subsection has been called ``strong exchange''.
When we review these properties below, we rename them all so we can talk about them unambiguously.

Borovik, Gelfand and White defined their strong exchange property in terms of hyperplane reflections of the Weyl group associated to $\Phi$, so we shall rename it the \emph{hyperplane exchange property}.
We will present only its specialisation to matroids, even $\Delta$-matroids and $\Delta$-matroids.
For the full definition, we refer to \cite{BorovikGelfandWhite2003,CDFS1}.

The \defn{hyperplane exchange property for (ordinary) matroids} is
\begin{equation}\label{axiom:strong-basis-ex}
    \text{For all distinct } A, B \in \matroid, \text{ there exist } a \in A \setminus B,\,b\in B\setminus A \text{ such that } A \setminus \{ a\} \cup \{b\}, B \setminus \{ b\} \cup \{a\} \in \matroid \, .
\end{equation}
Despite originally being named the strong exchange property, the hyperplane exchange property turns out to be equivalent to the basis exchange property for matroids.
In particular, it is shown in \cite[Theorem 1.5.2]{BorovikGelfandWhite2003} that a set system $\matroid$ satisfies the hyperplane exchange property if and only if it is a matroid.
Note that there are many other variants and strengthenings of the basis exchange property that all turn out to be equivalent.

The \defn{hyperplane exchange property for even $\Delta$-matroids} is% which was first defined by Wenzel in \cite{Wenzel1993}:
\begin{equation}\label{axiom:even-strong-symm-ex}
    \text{For all distinct } A,B \in \matroid, \text{ there exist distinct }a,\, b  \in A \symDiff B,\, \text{ such that } A \symDiff \{a,b\} \, , \, B \symDiff \{a,b\} \in M \, .
\end{equation}
As with matroids, this property in fact characterises even $\Delta$-matroids.
Furthermore, \defn{Wenzel's exchange property} --
\begin{equation}\label{axiom:wenzel-exchange}
    \text{For all } A,B \in \matroid, a \in A \symDiff B, \text{ there exists } b \in (A \symDiff B) \setminus a \text{ such that } A \symDiff \{a,b\} \, , \, B \symDiff \{a,b\} \in M \, .
\end{equation}
-- while apparently stronger, turns out to be equivalent, and also characterises even $\Delta$-matroids.
% As a corollary of Wenzel \cite[Theorem 1]{Wenzel1993}, it was shown in \cite[Theorem 2.9]{CDFS1} that a set system $M$ satisfies the strong exchange property for even $\Delta$-matroids if and only if is an even $\Delta$-matroid. \ben{Add type D equivalence theorem explicitly}
%Note that, although the paper \cite{Wenzel1993} is written in the generality of $\Delta$-matroids, 

\begin{theorem}[{\cite{Wenzel1993,CDFS1}}]
    \label{t:evenIsStrong}
    Let $M \subseteq [n]$. The following are equivalent:
    \begin{enumerate}
        \item $M$ is an even $\Delta$-matroid; \label{i:delta+1}
        \item $M$ satisfies the hyperplane exchange property for even $\Delta$-matroids; \label{i:delta+2}
        \item $M$ satisfies Wenzel's exchange property. \label{i:delta+3}
    \end{enumerate}
\end{theorem}
% Recall that a Coxeter matroid is strong if it satisfies the hyperplane exchange property.
% As a corollary of this result, it follows that every even $\Delta$-matroid is strong.

\noindent
Finally, we turn to $\Delta$-matroids.
The \defn{hyperplane exchange property for $\Delta$-matroids} is
\begin{equation}\label{axiom:strong-symm-ex}
    \text{For all distinct } A,B \in \matroid, \text{ there exist }a,\, b  \in A \symDiff B,\, \text{ such that } A \symDiff \{a,b\} \, , \, B \symDiff \{a,b\} \in M \, .
\end{equation}
We emphasise that \eqref{axiom:strong-symm-ex} again allows $a, b$ to be equal.
If a set system $M$ satisfies the hyperplane exchange property for $\Delta$-matroids, then it is a $\Delta$-matroid.
However, in contrast to matroids and even $\Delta$-matroids, the converse is not true as the following counterexample demonstrates.

\begin{example} \label{ex:delta+not+strong}
    Let $M$ be the following set system on $[3]$:
    \[
        M = \{\emptyset, 1, 2, 3, 123\} \, .
    \]
    The corresponding polytope $P(M)$ is displayed in \Cref{fig:3D-counterexample}, where checking the edge directions implies that $M$ is a $\Delta$-matroid.
    However, it does not satisfy the hyperplane exchange property \eqref{axiom:strong-symm-ex}.
    Taking $A = \emptyset$ and $B = 123$, we see there is no choice of $a$ and $b$ such that both $A \symDiff \{a,b\}$ and $B \symDiff \{a,b\}$ are contained in $M$, as at least one must have cardinality two.
\end{example}

\begin{figure}[ht]
    \centering
    \begin{tikzpicture}
        [
            scale=2.0,
            vertex/.style={inner sep=1.4pt,circle,draw=black,fill=black,thick},
        ]
        \coordinate (bl1) at (0,0);
        \coordinate (br1) at (1,0);
        \coordinate (tl1) at (0,1);
        \coordinate (tr1) at (1,1);

        \coordinate (bl2) at (0.5,0.8);
        \coordinate (br2) at (1.5,0.8);
        \coordinate (tl2) at (0.5,1.8);
        \coordinate (tr2) at (1.5,1.8);
        
        \draw[ultra thick] (bl1) -- (br1);
        \draw[ultra thick] (bl1) -- (tl1);
        \draw (br1) -- (tr1);
        \draw (tl1) -- (tr1);

        \draw[dashed] (bl2) -- (br2);
        \draw[dashed] (bl2) -- (tl2);
        \draw (br2) -- (tr2);
        \draw (tl2) -- (tr2);

        \draw[dashed, ultra thick] (bl1) -- (bl2);
        \draw (tl1) -- (tl2);
        \draw (br1) -- (br2);
        \draw (tr1) -- (tr2);

        \draw[ultra thick] (br1) -- (tr2);
        \draw[ultra thick] (tl1) -- (tr2);
        \draw[ultra thick, dashed] (bl2) -- (tr2);
        \draw[ultra thick, dashed] (tl1) -- (bl2);
        \draw[ultra thick, dashed] (br1) -- (bl2);
        \draw[ultra thick] (tl1) -- (br1);

        \node[below left] at (bl1) {$\emptyset$};
        \node[above right] at (tr2) {$123$};
        \node[vertex] at (bl1) {};
        \node[vertex] at (tr2) {};
        \node[vertex] at (br1) {};
        \node[vertex] at (tl1) {};
        \node[vertex] at (bl2) {};
        
        \fill[opacity=0.15] (bl1) -- (tl1) -- (tr2) -- (br1) -- (bl1);
    \end{tikzpicture}
    \caption{An example of a $\Delta$-matroid that is not a strong $\Delta$-matroid from \Cref{ex:delta+not+strong}.}
    \label{fig:3D-counterexample}
\end{figure}
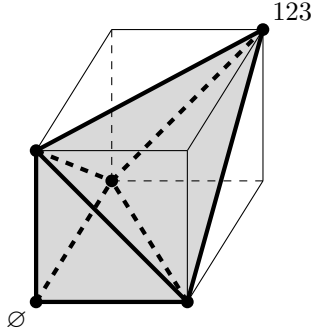

In light of this, we say a $\Delta$-matroid is \defn{strong} if it satisfies the hyperplane exchange property for $\Delta$-matroids.
A hierarchy for the various families of $\Delta$-matroids discussed is given in \Cref{fig:hierarchy}.
Recall that we do not use the prefix `strong' for matroids or for even $\Delta$-matroids, as both always satisfy their corresponding strong exchange properties \eqref{axiom:strong-basis-ex} and \eqref{axiom:even-strong-symm-ex}.

\begin{figure}[ht]
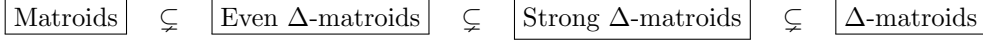

    \centering
    \begin{tabular}{ccccccc}
         \fbox{Matroids} & $\subsetneq$ & \fbox{Even $\Delta$-matroids} & $\subsetneq$ & \fbox{Strong $\Delta$-matroids} & $\subsetneq$ & \fbox{$\Delta$-matroids}
    \end{tabular}
    \caption{The hierarchy of the families of $\Delta$-matroids discussed.}
    \label{fig:hierarchy}
\end{figure}

We close this section with two alternative characterisation of strong $\Delta$-matroids.
As even $\Delta$-matroids have Wenzel's exchange property, $\Delta$-matroids have another exchange property that we call the \emph{weak Wenzel exchange property}:
%\footnote{Even this property sometimes appears in the literature as simply the ``strong exchange property''.}:
\begin{equation} \label{axiom:weak-wenzel}
    \text{For all } A,B \in \matroid, a \in A \symDiff B, \text{ there exist }b  \in A \symDiff B \, \text{ such that } A \symDiff \{a,b\} \, , \, B \symDiff \{a,b\} \in M \, .
\end{equation}
This is a weakening of Wenzel's exchange property where we do not require $a,b$ to be distinct.
We will show that for $\Delta$-matroids, the hyperplane exchange property and the weak Wenzel exchange property are in fact equivalent.
To do so, we will need a method of `lifting' properties of $\Delta$-matroids to even $\Delta$-matroids which we discuss now.

%We can characterise the $\Delta$-matroids that do satisfy the strong exchange property as follows.
Given a set $I \subseteq [n]$, we define the set $\overline{I} \subseteq [n] \cup \{*\}$ by
\[
\overline{I} = \begin{cases} I & |I| \text{ even,} \\ I \cup \{*\} & |I| \text{ odd.} \end{cases}
\]
Given a set system $M$ on $[n]$, we define $\overline{M} = \set{\overline{I}}{I \in M}$ to be a set system on $[n] \cup \{*\}$.
Clearly $\overline{M}$ is an even set system.

Our main result of this section is the following equivalent three characterisations of strong $\Delta$-matroids.

\begin{theorem} \label{thm:hyperplane-equals-wenzel}
Let $M$ be a $\Delta$-matroid.
The following are equivalent:
\begin{enumerate}
\item $M$ satisfies the hyperplane exchange property. \label{item:hew 1}
\item $M$ satisfies the weak Wenzel exchange property.\label{item:hew 2}
\item $\overline{M}$ is an even $\Delta$-matroid. \label{item:hew 3}
\end{enumerate}
\end{theorem}

The implication \eqref{item:hew 2} $\Rightarrow$ \eqref{item:hew 1} is clear.
The equivalence between \eqref{item:hew 2} and \eqref{item:hew 3} was proved in \cite[Theorem 3.1]{Murota2021}, which Murota himself credits to Geelen in 1996;
see also the discussion surrounding Theorem~2.25 in the recent \cite{DingKim}.
We will show below that \eqref{item:hew 1} $\Rightarrow$ \eqref{item:hew 3}, after establishing some lemmas.

\begin{lemma} \label{lem:restriction}
    Let $M$ be a $\Delta$-matroid on $[n]$.
    For any $I, J \subset [n]$, the set system
    \[
    M_{I,J} = \set{K \in M}{I \cap J \subseteq K \subseteq I \cup J} \, 
    \]
    is also a $\Delta$-matroid.
    Moreover, if $M$ satisfies the hyperplane exchange (resp.\ weak Wenzel) property, then $M_{I,J}$ does also.
\end{lemma}

\begin{proof}
    Let $A, B \in M_{I,J}$.
    As $I \cap J \subseteq A,B \subseteq I \cup J$, it follows that $A \symDiff B \subseteq (I \cup J) \setminus (I \cap J)$.
    Hence any $\sd{A}{a,b} \in M$ is also contained in $M_{I,J}$.
    It is now routine to check that if $M$ satisfies symmetric exchange, hyperplane exchange, or weak Wenzel exchange, then so does $M_{I,J}$.
\end{proof}

\begin{remark}
    A geometric viewpoint on $M_{I,J}$ is as the restriction of the polytope of $M$ to a face of the cube $P([n]) = [0,1]^n$.
    In particular, $P(M_{I,J}) = P(M) \cap \cube{I}{J}$ where
    \[
    \cube{I}{J} = \conv\set{\cubeVertex{S} \in \RR^n }{ I \cap J \subseteq S \subseteq I \cup J } \subseteq \RR^n \, 
    \]
    is the smallest face of $[0,1]^n$ that contains $\cubeVertex{I}$ and $\cubeVertex{J}$.
    This geometric viewpoint will be explored further in \Cref{s:antipodes} when we begin investigating antipodes on these polytopes.
\end{remark}

We next observe that for any two sets with opposite parity in a strong $\Delta$-matroid, we can perform a hyperplane exchange where $i=j$.
%Note this is Lemma 3.4 from \cite{CDFS1}.

\begin{lemma} \label{lem:odd+antipode+exchange}
    Let $M$ be a strong $\Delta$-matroid on $[n]$.
    If $I,J \in M$ are such that $|I \symDiff J| \equiv 1 \bmod 2$, then there exists $i \in I \symDiff J$ such that $I \symDiff \{i\}, J \symDiff \{i\} \in M$.
\end{lemma}

\begin{proof}
    We prove this by induction on $\order{I \symDiff J} = 2k + 1$.
    For the base cases of $k = 0, 1$, any hyperplane exchange suffices.

    For the inductive step where $k \geq 2$, by the hyperplane exchange condition we know there exists $a, b \in I \symDiff J$ such that $\sd{I}{a,b}$ and $\sd{J}{a,b}$ are in $\matroid$.
    If $a = b$, then we are done; therefore we may assume $a \neq b$.

    By induction, considering $\sd{I}{a,b}$ and $J$, there exists $c \neq a, b$ such that $\sd{I}{a,b,c}$ and $\sd{J}{c}$ are in $\matroid$.
    Similarly, considering $I$ and $\sd{J}{a,b}$, there exists $d \neq a, b$ such that $\sd{I}{d}$ and $\sd{J}{a, b, d}$ are in $\matroid$.
    Note that if $c = d$ then both $\sd{I}{d}$ and $\sd{J}{c}$ being in $\matroid$ implies we are done.
    Thus we may further assume $c \neq d$.

    By hyperplane exchange we know $I$ and $\sd{I}{a, b, c}$ being in $\matroid$ implies either: (A.1) $\sd{I}{c}$ and $\sd{I}{a,b}$ are in $\matroid$, (A.2) $\sd{I}{a}$ and $\sd{I}{b,c}$ are in $\matroid$, or (A.3) $\sd{I}{b}$ and $\sd{I}{a,c}$ are in $\matroid$.
    Similarly, $J,\, \sd{J}{a, b, d} \in \matroid$ implies either: (B.1) $\sd{J}{d}$ and $\sd{J}{a,b}$ are in $\matroid$, (B.2) $\sd{J}{a}$ and $\sd{J}{b,d}$ are in $\matroid$, or (B.3) $\sd{J}{b}$ and $\sd{J}{a,d}$ are in $\matroid$.
    
    In case (A.1) we are done, since this case provides $\sd{I}{c} \in \matroid$ and our setup gives $\sd{J}{c} \in \matroid$.
    Likewise we are done in case (B.1).
    Note further that if we have (A.2) and (B.2) then we are also done since $\sd{I}{a}$ and $\sd{J}{a}$ would be in $\matroid$; and similarly with (A.3) and (B.3).
    Therefore, without loss of generality, we may assume we have (A.2) and (B.3).

    Up to now we have the following elements in $\matroid$:
    \begin{gather*}
        I,\, \sd{I}{a},\, \sd{I}{d},\, \sd{I}{a,b},\, \sd{I}{b,c},\, \sd{I}{a,b,c},\\
        J,\, \sd{J}{b},\, \sd{J}{c},\, \sd{J}{a,b},\, \sd{J}{a,d},\, \sd{J}{a,b,d}\\
        \text{such that } a, b, c \text{ and } d \text{ are pairwise distinct.}
    \end{gather*}
    By induction, $I$ and $\sd{J}{a,d}$ must have an element $e$ (equal to neither $a$ nor~$d$) in their symmetric difference such that $\sd{I}{e}$ and $\sd{J}{a, d, e}$ are in $\matroid$.
    If $e = b$ or $c$, then since $\sd{J}{b}$ and $\sd{J}{c}$ are in $\matroid$, we are done; therefore we may assume $e \neq b, c$.
    %Similarly, by induction, $J$ and $\sd{I}{b,c}$ must contain an $f$ (not equal to $b$ nor $c$) in their symmetric difference such that $\sd{I}{b, c, f}$ and $\sd{J}{f}$ are in $\matroid$, and by a similar argument as before, we may assume $f \neq a, d$.

    %Finally, taking $I$ and $\sd{I}{b,c,f}$, we know either $\sd{I}{b}$, $\sd{I}{c}$ or $\sd{I}{f}$ are in $\matroid$. But since $\sd{J}{b}$, $\sd{J}{c}$ and $\sd{J}{f}$ are already in $\matroid$, we are done.
    Finally, taking $J$ and $\sd{J}{a,d,e}$, we know either $\sd{J}{a}$, $\sd{J}{d}$ or $\sd{J}{e}$ are in $\matroid$. But since $\sd{I}{a}$, $\sd{I}{d}$ and $\sd{I}{e}$ are already in $\matroid$, we are done with the proof.
    %    Similarly, we have the same for $\sd{J}{a, d, e}$.
    %Thus we are done with the proof.
\end{proof}

\begin{proof}[Proof of \Cref{thm:hyperplane-equals-wenzel}]
As stated earlier, the only implication that remains to be proved is \eqref{item:hew 1} $\Rightarrow$ \eqref{item:hew 3}.
Let $M$ be a $\Delta$-matroid on $[n]$ satisfying the hyperplane exchange property; we show that $\overline{M}$ is an even $\Delta$-matroid on $[n] \cup \{*\}$.
We proceed by induction on $n$.
In the base cases $n=0,1,2$, every $\Delta$-matroid satisfies the weak Wenzel exchange property, and so the equivalence \eqref{item:hew 2} $\Leftrightarrow$ \eqref{item:hew 3} gives the result.

For the inductive step, we assume that for any $n' < n$, a $\Delta$-matroid $N$ on $[n']$ satisfying the hyperplane exchange property can be lifted to an even $\Delta$-matroid $\overline{N}$, and hence also satisfies the weak Wenzel exchange property.
By \Cref{lem:restriction}, the set system $M_{I,J}$ is also a $\Delta$-matroid satisfying the hyperplane exchange property for any $I,J \subseteq [n]$.
% Recall that if $M$ is a $\Delta$-matroid, then $M' = M \cap \cube{I}{J}$, the restriction of $M$ to the hypercube $\cube{I}{J}$, is also a $\Delta$-matroid.
% Moreover, if $M$ satisfies Hyperplane Exchange then $M'$ must also satisfy Hyperplane Exchange. \aram{Might be good to mention/link to props?}
Up to relabelling, $M_{I,J}$ is a set system on $[n']$ where $n' = |I \symDiff J|$, hence by the induction hypothesis $M_{I,J}$ satisfies the weak Wenzel exchange property whenever $|I \symDiff J| < n$.
As such, we can assume that for any $I, J \in M$ with $|I \symDiff J| < n$, we can apply a weak Wenzel exchange.

Let $\overline{I}, \overline{J} \in \overline{M}$. We will show there exist distinct $\{i,j\} \in \overline{I} \symDiff \overline{J}$ such that $\overline{I} \symDiff \{i,j\}$ and $\overline{J} \symDiff \{i,j\} \in \overline{M}$.
If $|I \symDiff J| < n$, then by induction $M_{I,J}$ can be lifted to the even $\Delta$-matroid $\overline{M_{I,J}}$, and so there exists $\overline{I} \symDiff \{i,j\}$ and $\overline{J} \symDiff \{i,j\} \in \overline{M_{I,J}} \subset \overline{M}$.
As such, it suffices to only prove the case where $|I \symDiff J| = n$.

A useful fact we will use throughout is that for any subsets $A,B \in M$, we have $\overline{A \symDiff B} = \overline{A} \symDiff \overline{B}$.
    To see this, note that $* \in \overline{A \symDiff B}$ if and only if $|A \symDiff B|$ is odd, or equivalently $|A|$ and $|B|$ have different parity.
    This is equivalent to $* \in \overline{A} \symDiff \overline{B}$.
    
First consider the case where $n$ is odd.
By hyperplane exchange, there exists (possibly non-distinct) $i,j \in I \symDiff J$ such that $I \symDiff \{i,j\}$ and $J \symDiff \{i,j\} \in M$.
Then
\begin{equation} \label{eq:lift+exchange}
\overline{\sd{I}{i,j}} = \sd{\overline{I}}{\overline{i,j}} = \begin{cases}
\sd{\overline{I}}{i,j} \in \overline{M} & i\neq j \\ \sd{\overline{I}}{i,*} \in \overline{M} & i=j
\end{cases}
\end{equation}
and similarly for $\sd{\overline{J}}{\overline{i,j}}$.
As $|I\symDiff J| = n$ is odd, it follows that $\overline{I} \symDiff \overline{J} = (I \symDiff J) \cup \{*\}$ and so either case above is a valid exchange between $\overline{I}$ and $\overline{J}$.
Therefore $\overline{M}$ satisfies the hyperplane exchange property (for even $\Delta$-matroids).
By \Cref{t:evenIsStrong}, $\overline{M}$ is an even $\Delta$-matroid.

We next consider the case where $n$ is even, the far more intricate case.
By hyperplane exchange, there exists (possibly non-distinct) $i,j \in I \symDiff J$ such that $I \symDiff \{i,j\}$ and $J \symDiff \{i,j\} \in M$.
%If $i,j$ are distinct, then the first case of \eqref{eq:lift+exchange} suffices.
If $i=j$, repeating the second case of \eqref{eq:lift+exchange} does not work as $* \notin \overline{I} \symDiff \overline{J}$.
We show that we can always find another distinct element $j \in I \symDiff J \setminus \{i\}$ such that $I \symDiff \{i,j\},\, J \symDiff \{i,j\} \in M$.

\begin{claim} \label{lem:even+antipode+exchange}
    %Let $M \subseteq [n]$ be a strong $\Delta$-matroid.
    If $I,J \in M$ are such that $|I \symDiff J| = n = 2m$ for some $m \in \ZZ$ 
    %then there exists distinct $i,j \in I \symDiff J$ such that $I \symDiff \{i,j\}, J \symDiff \{i,j\} \in M$.
    and $i \in I \symDiff J$ such that $I \symDiff \{i\},\, J \symDiff \{i\} \in M$, then there exists some $j \in I \symDiff J \setminus \{i\}$ such that $I \symDiff \{i,j\},\, J \symDiff \{i,j\} \in M$.
\end{claim}

\begin{proof}
    We prove this via induction on $m$. %$|I \symDiff J| = 2m$.
    Note that we will continue to use the induction hypothesis in the proof of \Cref{thm:hyperplane-equals-wenzel}, hence can apply weak Wenzel to any sets $A,B$ with $|A \symDiff B| < n = 2m$.
    For the base case of $m=1$, we have $I \symDiff J = \{i,j\}$, and so $J = I \symDiff \{i,j\}$ and vice versa.
    Hence the claim trivially holds.

    For the general case, %by Hyperplane Exchange there exists possibly non-distinct $i,i'$ such that $\sd{I}{i,i'}$ and $\sd{J}{i,i'} \in M$.
    %We're done if they are distinct, so assume $i = i'$.
    applying \Cref{lem:odd+antipode+exchange} to $I\symDiff \{i\}$ and $J$ implies there exists some $j \in (I \symDiff J) \setminus \{i\}$ such that $I \symDiff \{i,j\}$ and $J \symDiff \{j\}$ are in $M$.
    As $I \symDiff (J \symDiff \{j\})$ has cardinality less than $n$, %by induction hypothesis of \Cref{thm:hyperplane-equals-wenzel}, 
    we can apply weak Wenzel to $I$ and $J \symDiff \{j\}$ for $i \in I \symDiff (J \symDiff \{j\})$, obtaining an element $k$.
    Either $k = i$ giving $J \symDiff \{i,j\} \in M$, or else $k \in I \symDiff J \setminus \{i,j\}$ is such that $I \symDiff \{i,k\}$ and $J \symDiff \{i,j,k\}$ are both in $M$.
    In the former case we are done, so let us assume the latter case.
    Applying hyperplane exchange to $J$ and $J \symDiff \{i,j,k\}$, we have either $J \symDiff \{i,j\}$, $J \symDiff \{i,k\}$ or $J \symDiff\{j,k\}$ are in $M$.
    If either of the former two hold then we are done, so assume the final case.

    Consider $I\symDiff \{i\}$ and $J \symDiff\{i,j,k\}$.
    Observe that $(I\symDiff \{i\}) \symDiff (J \symDiff\{i,j,k\})$ has cardinality $2(m-1)$, contains $i$, and that
    \[
    I = (I\symDiff \{i\}) \symDiff \{i\} \in M \, , \quad J \symDiff\{j,k\} = (J \symDiff\{i,j,k\}) \symDiff \{i\} \in M \, .
    \]
    Hence we can apply the induction hypothesis of the claim, and there exists some $\ell \in (I \symDiff J) \setminus \{i,j,k\}$ such that $I \symDiff \{i,\ell\}$ and $J \symDiff \{i,j,k,\ell\}$ are both in $M$.
    Finally, applying hyperplane exchange to $J \symDiff \{i,j,k,\ell\}$ and $J\symDiff \{i\}$ implies one of $J \symDiff \{i,j\}$, $J \symDiff \{i,k\}$ or $J \symDiff \{i,\ell\}$ are in $M$.
    In all cases, we are done.
\end{proof}
Returning to the proof of \Cref{thm:hyperplane-equals-wenzel},
\Cref{lem:even+antipode+exchange} implies that we can always apply a hyperplane exchange between $I,J$ where $i,j \in I \symDiff J$ are distinct.
Thus in \eqref{eq:lift+exchange} we may choose to be in the first case, giving us $\sd{\overline{I}}{i,j}$ and $\sd{\overline{J}}{i,j} \in \overline{M}$ with $i,j \in \overline{I} \symDiff \overline{J}$, and completing the proof.
\end{proof}
% All the work

% \alexIL{Why is this remark here?}
% \benIL{Don't remember, remove}
% \begin{remark}
%     In \cite{CDFS1}, the authors gave characterisations of strong Coxeter matroids for all minuscule parabolic subgroups, which includes $\Delta$-matroids and even $\Delta$-matroids.
%     These characterisations were in terms of `strong exchange equations', tropical equations whose solutions were exactly the Coxeter matroids of a given type and parabolic subgroup.
% \end{remark}

\Cref{thm:hyperplane-equals-wenzel} relates strong $\Delta$-matroids and even $\Delta$-matroids via lifting.
But the proof also uncovers another relationship: restricting a strong $\Delta$-matroid to subsets of a certain parity yields an even $\Delta$-matroid.

\begin{corollary} \label{cor:strong+factors+into+even}
Let $M$ be a strong $\Delta$-matroid.
Then
\[
M_+ := \set{A \in M}{|A| \equiv 0 \bmod 2} \, , \quad M_- := \set{A \in M}{|A| \equiv 1 \bmod 2}
\]
are even $\Delta$-matroids.
\end{corollary}
\begin{proof}
    \Cref{lem:even+antipode+exchange} implies that if $M$ satisfies the hyperplane exchange property, then for any $I, J$ with $|I \symDiff J| \equiv 0 \bmod 2$ there exist distinct $i,j \in I \symDiff J$ such that $I \symDiff \{i,j\}$ and $J \symDiff \{i,j\} \in M$.
    In particular, the restrictions $M_+$ and $M_-$ satisfy the hyperplane exchange property for even $\Delta$-matroids, completing the proof.
\end{proof}

Note that the converse of \Cref{cor:strong+factors+into+even} is not true: the $\Delta$-matroid $M$ from \Cref{ex:delta+not+strong} is not strong, but $M_+$ and $M_-$ are even $\Delta$-matroids.

% \begin{example}\label{ex:standard}
%     We return to the $\Delta$-matroid $M = \{\emptyset, 1, 2, 3, 123\}$ of Example~\ref{ex:delta+not+strong}, which we saw does not satisfy strong exchange.
%     This lifts to $\overline{M} = \{\emptyset, 1*, 2*, 3*, 123*\}$, all of which have the same parity.
%     This is not an even $\Delta$-matroid as $* \in \emptyset \symDiff 123*$, but no other element can be removed from $123$ to give a set in $\overline{M}$.
% \end{example}

\section{Peerless antipodes and strong $\Delta$-matroids} \label{s:antipodes}

%\subsubsection{Antipodes}
%\alex{Annoying thing which we should take care with:
%The cube $[0,1]^n$ is not a $W$ orbit polytope.
%From the reprsentation-theoretic point of view it is in principle important which cube we take:
%in particular, the spinor variety is defined by Pfaffians
%because its Grassmann embedding corresponds to the convex hull of the orbit of weights $W\cdot(\frac12,\ldots,\frac12)$ rather than $W\cdot(1,\ldots,1)$.
%Perhaps the right move is to redefine the symbol ``$e_I$'':
%but I don't know how that will harm the representation theory. Kieran?} \KC{Essentially, I have never referred to where the cube is located. I have exclusively used $I$ a subset of $[n]$ and then defined a basis $e_i$ and $e_I = e_{i_1}\wedge \ldots \wedge e_{i_k}$. Essentially, I will need to check but I have never used the $[0,1]^n$ description} \KC{Actually tell a lie, I have explicitly (or perhaps implicitly) used the cube to be $[-1/2,1/2]^n$ Because I have been assuming it is the convex hull of the weights (all of which are $(\pm \frac{1}{2})^n$)}
%\alex{Good! Then we just need to figure out how to distinguish the exterior algebra element and the polytope vertex that are both called $e_I$ at the moment.
%I'll make a command \texttt{\textbackslash cubevertex} for when we mean a vertex of a cube.}

Given $I,J \subseteq [n]$, we define $\cube{I}{J}$ to be the hypercube
\[
\cube{I}{J} = \conv\set{\cubeVertex{S} \in \RR^n }{ I \cap J \subseteq S \subseteq I \cup J } \subseteq \RR^n \, .
\]
More generally, we say a polytope $Q \subseteq \RR^n$ is a $k$-cube if $Q = \cube{I}{J}$ for some $I,J \subseteq [n]$ where $|I \symDiff J| = k$.
Note that $\cube{I}{J}$ is the smallest face of the `ambient cube' $\cube{\emptyset}{[n]}$ that contains both $\cubeVertex{I}$ and $\cubeVertex{J}$.
%of the ambient cube $\ambientPoly=\conv\{\cubeVertex{K}:K\subseteq[n]\}$ for $\Delta$-matroids. 
The dimension of $\cube{I}{J}$ is $|I \symDiff J|$.
Moreover, as $I \cup J = (I \cap J) \cup (I \symDiff J)$, any sets $S, T$ with $S \cap T = I \cap J$ and $S \symDiff T = I \symDiff J$ define the same hypercube $\cube{S}{T} = \cube{I}{J}$.
For such $S$ and $T$, we call the pair an $(\cubeVertex{S}, \cubeVertex{T})$ \defn{antipode} of the cube $\cube{I}{J}$, as the two vertices in it are antipodal in the geometric sense.

\begin{definition}\label{def:peerless+isolated}
    Let $M$ be a set system of $[n]$ and let $I,J \in M$.
    \begin{itemize}
        \item We say $(\cubeVertex{I}, \cubeVertex{J}) \subseteq \polyOf{M}$ is a \defn{peerless antipode} (in $\cube{I}{J}$) if it is the only antipode of $\cube{I}{J}$ contained in $\polyOf{M}$.
        Equivalently, $(I,J)$ is a peerless antipode of $M$ if $S,\,T \in M$ with $S \cap T = I \cap J$ and $S \symDiff T = I \symDiff J$ implies $S=I$ and $T=J$, or vice versa.
        \item We say $(\cubeVertex{I}, \cubeVertex{J}) \subseteq \polyOf{M}$ is an \defn{isolated antipode} (in $\cube{I}{J}$) if the restriction of $\polyOf{M}$ to $\cube{I}{J}$ is exactly $(\cubeVertex{I}, \cubeVertex{J})$. 
        Equivalently, $(I,J)$ is an isolated antipode of $M$ if $S \in M$ and $I \cap J \subseteq S \subseteq I \cup J$ implies that $S=I$ or $S=J$.
    \end{itemize}
\end{definition}
\noindent In particular, isolated antipodes are necessarily peerless, but the converse is not always true.

If we call $(\cubeVertex{I},\, \cubeVertex{J})$ simply an antipode of $P(M)$,
or of~$M$ as we will do interchangeably,
we mean to see it as an antipode in the cube $\cube{I}{J}$ it determines.

\begin{example}
\label{ex:antipodes}
    We give examples of peerless and isolated antipodes when $\cube{I}{J}$ is a $3$-cube.
    The cubes below represent the hypercube $\cube{I}{J}$ with the subpolytope of $\polyOf{M}$ being represented by the bold lines.
    The vertices are the vertices present in $\cube{I}{J} \cap \polyOf{M}$.
    On the left, we give an example of an antipode which is peerless and not isolated and on the right we give an example of an antipode which is isolated.
    
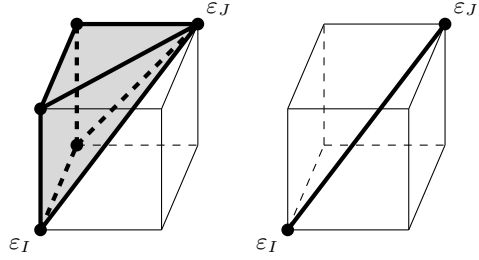
\begin{center}
        \begin{tikzpicture}
            [
                scale=1.6,
                vertex/.style={inner sep=1.4pt,circle,draw=black,fill=black,thick},
            ]
            \coordinate (bl1) at (0,0);
            \coordinate (br1) at (1,0);
            \coordinate (tl1) at (0,1);
            \coordinate (tr1) at (1,1);

            \coordinate (bl2) at (0.3,0.7);
            \coordinate (br2) at (1.3,0.7);
            \coordinate (tl2) at (0.3,1.7);
            \coordinate (tr2) at (1.3,1.7);
            
            \draw (bl1) -- (br1);
            \draw[ultra thick] (bl1) -- (tl1);
            \draw (br1) -- (tr1);
            \draw (tl1) -- (tr1);

            \draw[dashed] (bl2) -- (br2);
            \draw[dashed, ultra thick] (bl2) -- (tl2);
            \draw (br2) -- (tr2);
            \draw[ultra thick] (tl2) -- (tr2);

            \draw[dashed, ultra thick] (bl1) -- (bl2);
            \draw[ultra thick] (tl1) -- (tl2);
            \draw (br1) -- (br2);
            \draw (tr1) -- (tr2);

            \draw[ultra thick] (bl1) -- (tr2);
            \draw[ultra thick] (tl1) -- (tr2);
            \draw[ultra thick, dashed] (bl2) -- (tr2);

            \node[below left] at (bl1) {$\cubeVertex{I}$};
            \node[above right] at (tr2) {$\cubeVertex{J}$};
            \node[vertex] at (bl1) {};
            \node[vertex] at (tr2) {};
            \node[vertex] at (tl1) {};
            \node[vertex] at (tl2) {};
            \node[vertex] at (bl2) {};
            \fill[opacity=0.15] (bl1) -- (tr2) -- (tl2) -- (tl1) -- (bl1);
        \end{tikzpicture}
        \begin{tikzpicture}
            [
                scale=1.6,
                vertex/.style={inner sep=1.4pt,circle,draw=black,fill=black,thick},
            ]
            \coordinate (bl1) at (0,0);
            \coordinate (br1) at (1,0);
            \coordinate (tl1) at (0,1);
            \coordinate (tr1) at (1,1);

            \coordinate (bl2) at (0.3,0.7);
            \coordinate (br2) at (1.3,0.7);
            \coordinate (tl2) at (0.3,1.7);
            \coordinate (tr2) at (1.3,1.7);
            
            \draw (bl1) -- (br1);
            \draw (bl1) -- (tl1);
            \draw (br1) -- (tr1);
            \draw (tl1) -- (tr1);

            \draw[dashed] (bl2) -- (br2);
            \draw[dashed] (bl2) -- (tl2);
            \draw (br2) -- (tr2);
            \draw (tl2) -- (tr2);

            \draw[dashed] (bl1) -- (bl2);
            \draw (tl1) -- (tl2);
            \draw (br1) -- (br2);
            \draw (tr1) -- (tr2);

            \draw[ultra thick] (bl1) -- (tr2);

            \node[below left] at (bl1) {$\cubeVertex{I}$};
            \node[above right] at (tr2) {$\cubeVertex{J}$};
            \node[vertex] at (bl1) {};
            \node[vertex] at (tr2) {};
        \end{tikzpicture}
    \end{center}
\end{example}
%\delete{If it does contain exactly one antipode, we call such an antipode a \defn{peerless antipode}.}
%\aram{If it does contain exactly one antipode, this antipode is peerless.}

Given two vertices of $P(M)$, one can ask whether the antipode they form is peerless or isolated;
but a more obvious question is whether the two vertices span an edge.
In fact, being an edge is intermediate in strength between our two properties of antipodes.

\begin{lemma}\label{lem:edge-implies-peerless}
Let $(\cubeVertex{I},\cubeVertex{J})$ be an edge of $\polyOf{M}$.
Then $(\cubeVertex{I},\cubeVertex{J})$ is a peerless antipode of $\polyOf{M}$ (in $\cube{I}{J}$).
\end{lemma}
\begin{proof}
    Let $\ell\colon \RR^n \rightarrow \RR$ be a linear functional that is maximised uniquely on the edge $(\cubeVertex{I},\cubeVertex{J})$, i.e. $\ell(\cubeVertex{I}) = \ell(\cubeVertex{J}) > \ell(\cubeVertex{K})$ for all other vertices $\cubeVertex{K}$ of $\polyOf{M}$.
    Suppose $(\cubeVertex{S},\cubeVertex{T})$ is another antipode in $\cube{I}{J}$. Then $S \cap T = I \cap J$ and $S \cup T = I \cup J$.
    By linearity of $\ell$, we have
    \[
    \ell(\cubeVertex{S}) + \ell(\cubeVertex{T}) = \ell(\cubeVertex{S \cap T}) + \ell(\cubeVertex{S \cup T}) = \ell(\cubeVertex{I \cap J}) + \ell(\cubeVertex{I \cup J}) = \ell(\cubeVertex{I}) + \ell(\cubeVertex{J}) \, ,
    \]
    which contradicts that $\ell$ is strictly maximised on $\cubeVertex{I}$ and $\cubeVertex{J}$.
\end{proof}
\begin{lemma}\label{lem:isolated-implies-edge}
    Let $(\cubeVertex{I},\cubeVertex{J})$ be an isolated antipode of $\polyOf{M}$ (in $\cube{I}{J}$).
    Then $(\cubeVertex{I},\cubeVertex{J})$ is an edge of $\polyOf{M}$.
\end{lemma}
\begin{proof}
    Define a linear functional $\ell$ by
    \begin{equation} \label{eq:cube+functional}
    \ell(e_i) =
    \begin{cases}
        1 & i \in I \cap J \\
        0 & i \in I \symDiff J \\
        -\alpha & i \notin I \cup J
    \end{cases}
    \end{equation}
    for some $\alpha > 0$. 
    The face of $\cube{\emptyset}{[n]}$ maximising $\ell$ is $\cube{I}{J}$.
    Of its vertices, by assumption only $\cubeVertex{I}$ and $\cubeVertex{J}$ are in $\polyOf{M}$.
    As all vertices of $\polyOf{M}$ are vertices of~$\cube{\emptyset}{[n]}$,
    we conclude that $\cubeVertex{I}$ and $\cubeVertex{J}$ are the only vertices of $\polyOf{M}$ on which $\ell$ is maximised.
    The $\ell$-maximising face of $\polyOf{M}$ is therefore their convex hull, which is the required edge.
\end{proof}

Our goal is to characterise families of $\Delta$-matroids in terms of antipodes.
We begin by observing two sufficient conditions for a set system to be a $\Delta$-matroid.

\begin{proposition}\label{prop:antipode-implies-delta-matroid}
    Let $M$ be a set system with no peerless antipode in any $k$-cube for $k\geq 3$.
    Then $M$ is a $\Delta$-matroid.
\end{proposition}
\begin{proof}
    Lemma~\ref{lem:edge-implies-peerless} implies that edges $(\cubeVertex{I}, \cubeVertex{J})$ of $P(M)$ can only occur for sets with $|I \symDiff J| \leq 2$.
    Hence all edges of $P(M)$ are parallel to $e_i \pm e_j$ or $e_i$ for some $i,j$, making $M$ a $\Delta$-matroid.
\end{proof}

The implication holds even with weaker hypotheses on the cubes of dimension $\ge5$.

\begin{proposition}\label{prop:antipode-implies-delta-matroid-2}
    Let $M$ be a set system.
    Suppose
    \begin{itemize}
        \item $M$ has no peerless antipode in any $3$-cube or $4$-cube, and
        \item $M$ has no isolated antipode in any $k$-cube for $k \geq 5$.
    \end{itemize}
    Then $M$ is a $\Delta$-matroid.
\end{proposition}
\begin{proof}
    Let $M$ be a set system with no peerless antipode in any $3$- or $4$-cube or isolated antipode in any $k$-cube for $k \geq 5$.
    We show by induction on $k$ that there are no edges $(\cubeVertex{I}, \cubeVertex{J})$ with $|I \symDiff J| = k$ for $k \geq 3$.
    The base cases of $k=3, 4$ are immediately ruled out by Lemma~\ref{lem:edge-implies-peerless}.
    
    For the inductive step we proceed by contradiction.
    Suppose that there are no edges of `exchange length' $\{3, \dots, k-1\}$ but there exists some edge $(\cubeVertex{I}, \cubeVertex{J})$ with $|I \symDiff J| = k$.
    Consider the face $G = P(M) \cap \cube{I}{J}$ of $P(M)$ on which the linear functional \eqref{eq:cube+functional} is maximised.
    As $(\cubeVertex{I}, \cubeVertex{J})$ is not isolated, we must have $\dim(G) > 1$.
    Therefore, there exists some two-dimensional face $F \subseteq G$ of $\polyOf{M}$ that contains $(\cubeVertex{I}, \cubeVertex{J})$ as an edge.
    Every other edge of $F$ must be strictly shorter than $(\cubeVertex{I}, \cubeVertex{J})$: 
    any edge of the same length would also be an antipode of $\cube{I}{J}$, but this second antipode would have the same midpoint as $\conv\{\cubeVertex{I}, \cubeVertex{J}\}$, and one edge of a polytope cannot intersect another in an interior point.
    By the induction hypothesis, this implies that all other edges of $F$ are parallel to some $e_i$ or $e_i \pm e_j$.
    However, it takes at least $\ceil{ k/2 }$ affinely independent edges of the form $e_i$ and $e_i \pm e_j$ to walk from $\cubeVertex{I}$ to $\cubeVertex{J}$, implying that $\dim(F) \geq \ceil{ k/2} > 2$, giving a contradiction.
\end{proof}

The converse of these propositions are not true as \Cref{ex:delta+not+strong} demonstrates.
As we shall see, these conditions precisely characterise strong $\Delta$-matroids.

\subsection{Antipode cryptomorphisms for even $\Delta$-matroids}\label{ssec:even-antipodes}

We give two characterisations of even $\Delta$-matroids among even set systems in terms of peerless and isolated antipodes.
As just discussed, antipodes in 2-cubes are `unproblematic' and represent valid exchanges.
Banning all peerless antipodes in larger hypercubes characterises even $\Delta$-matroids (\Cref{prop:type+1+D}).
It is also enough to ban peerless antipodes locally in the smallest problematic cubes, namely 4-cubes, and isolated antipodes in larger hypercubes (\Cref{prop:type+2+D}).

\begin{proposition} \label{prop:type+1+D}
    Let $M$ be an even set system.
    Then $M$ is an even $\Delta$-matroid if and only if it has no peerless antipode in any $2k$-cube for $k\geq 2$.
\end{proposition}
\begin{proof}
% Let $(\cubeVertex{I},\cubeVertex{J})$ be an edge of $\polyOf{M}$.
% Note that as $I,J$ have the same parity, $|I \symDiff J| = 2k$ for some $k$.
% Hence $(\cubeVertex{I},\cubeVertex{J})$ is an antipode of the $2k$-cube $\cube{I}{J}$.
% If $M$ has only peerless antipodes in $2$-cubes, then Lemma~\ref{lem:edge-implies-peerless} implies that $k = 1$, and hence the edge $(\cubeVertex{I},\cubeVertex{J})$ is parallel to $e_i \pm e_j$ where $\{i,j\} = I \symDiff J$.
Suppose $M$ has no peerless antipode in any $2k$-cube for $k\geq 2$.
As $M$ is an even set system, it also has no antipodes in any $\ell$-cube for $\ell \geq 3$ odd.
By \Cref{prop:antipode-implies-delta-matroid}, this implies that $M$ is an even $\Delta$-matroid.

Conversely, suppose that $M$ is an even $\Delta$-matroid, and $(\cubeVertex{I}, \cubeVertex{J}) \in \polyOf{M}$ is an antipode where $|I \symDiff J| \geq 4$.
By the weak Wenzel exchange property, for any $i \in I \symDiff J$ there exists $j \in J \symDiff I$ such that $I' = I \symDiff \{i, j\}$ and $J' = J \symDiff \{i, j\} \in M$. 
The condition $|I \symDiff J| \geq 4$ ensures that $I' \neq J$ and $J' \neq I$, and hence $(\cubeVertex{I'}, \cubeVertex{J'})$ is another distinct antipode in $\polyOf{M} \cap \cube{I}{J}$.
\end{proof}

\begin{proposition} \label{prop:type+2+D}
    Let $M$ be an even set system.
    Then $M$ is an even $\Delta$-matroid if and only if 
    \begin{itemize}
        \item it has no peerless antipode in any $4$-cube, and
        \item it has no isolated antipode in any $2k$-cube for $k > 2$.
    \end{itemize}
\end{proposition}
\begin{proof}
    If $M$ is an even $\Delta$-matroid, by Proposition~\ref{prop:type+1+D} it has no peerless antipodes in any $2k$-cube for $k \geq 2$.
    This automatically implies no $2k$-cube contains an isolated antipode.

    For the converse direction, observe that as $M$ is even, $M$ has no antipodes in any $\ell$-cube for $\ell \geq 3$ odd.
    In particular, it has no peerless antipode in any $3$-cube or $4$-cube, and no isolated antipode in any $k$-cube for $k\geq 4$.
    Hence \Cref{prop:antipode-implies-delta-matroid-2} implies $M$ is an even $\Delta$-matroid.
    %
    % Conversely, suppose $M$ satisfies the peerless and isolated antipodes conditions in the claim, but is not an even $\Delta$-matroid.
    % Then there exists some edge $(\cubeVertex{I}, \cubeVertex{J})$ with $|I \symDiff J| = 2k$ for $k \geq 2$, assume it is an edge of minimal length with this condition.
    % Lemma~\ref{lem:edge-implies-peerless} immediately rules out the case that $k=2$, and so we can assume $k>2$.
    % As $(\cubeVertex{I}, \cubeVertex{J})$ is not isolated, there exists a two-dimensional face $F \subseteq \cube{I}{J}$ of $\polyOf{M}$ that contains $(\cubeVertex{I}, \cubeVertex{J})$ as an edge.
    % As $(\cubeVertex{I}, \cubeVertex{J})$ is peerless by Lemma~\ref{lem:edge-implies-peerless}, the other edges in $F$ are necessarily shorter, and so by our assumption all parallel to some $e_i \pm e_j$.
    % However, it takes at least $k$ affinely independent edges of the form $e_i \pm e_j$ to walk from $\cubeVertex{I}$ to $\cubeVertex{J}$, implying that $\dim(F) \geq k >2$, giving a contradiction.
    % \alex{TODO: be a bit more careful with the induction: why can't there be an even longer edge in the 2-face?}
\end{proof}

\subsection{Antipode cryptomorphisms for matroids}\label{ssec:matroids-antipodes}

We next restrict the results on even $\Delta$-matroids to matroids, obtaining a new cryptomorphism for matroids in terms of a local exchange axiom.

To do so, we recall a small amount of polyhedral geometry.
% Given integers $0 \leq k \leq n$, we recall the $(n,k)$-hypersimplex is
% \begin{align*}
% \cH_{n,k} = \conv\set{\cubeVertex{S} \in \RR^n}{|S| = k \, , \, S \subseteq [n]} \\    
% &= [0,1]^n \cap \set{x \in \RR^n}{\sum x_i = k}
% \end{align*}
% This is an $(n-1)$-dimensional polytope.
Suppose $I,J \subseteq [n]$ with $|I| = |J|$.
We define the hypersimplex $\hypersimplex{I}{J}$ to be
\begin{align*}
\hypersimplex{I}{J} &= \conv\set{\cubeVertex{S} \in \RR^n}{|S| = |I| \, , \, I \cap J \subseteq S \subseteq I \cup J} \, \\
&= \cube{I}{J} \cap \set{x \in \RR^n}{\sum x_i = |I|} \, .
\end{align*}
Note that as $|I| = |J|$, we must have $|I \symDiff J| = 2k$ for some $k \in \ZZ_{\geq 0}$.
As such, $\cube{I}{J}$ is a $2k$-cube and so $\hypersimplex{I}{J}$ has dimension $2k-1$.
More generally, we say a polytope $Q \subseteq \RR^n$ is a \defn{$(k,2k)$-hypersimplex} if $Q = \hypersimplex{I}{J}$ for some $I,J \subseteq [n]$ with $|I| = |J|$ and $|I \symDiff J| = 2k$.
We use this terminology as it is not hard to see that $\hypersimplex{I}{J}$ is affinely isomorphic to the standard $(k,2k)$-hypersimplex
\[
\hypersimplex{k}{2k} = \conv\set{\cubeVertex{S} \in \RR^{2k}}{|S| = k \, , \, S \subseteq [2k]}
\]
via the map that sends each vertex $\cubeVertex{S}$ of $\hypersimplex{I}{J}$ to the vertex $\cubeVertex{S \setminus (I \cap J)}$ of $\hypersimplex{k}{2k}$.

\Cref{def:peerless+isolated} applies without change to matroids, defining peerless and isolated antipodes.
Note that if $|I|=|J|$ then the points $\cubeVertex{I}$ and~$\cubeVertex{J}$,
geometrically antipodal within $\cube{I}{J}$, are also antipodal within $\hypersimplex{I}{J}$, and conversely.
So when working with a matroid we will speak of antipodes in hypersimplices, rather than in cubes.

As a direct corollary of \Cref{prop:type+1+D}, we get the following known characterisation.

\begin{corollary}\label{cor:type+1+A}
    Let $M$ be an equicardinal set system.
    Then $M$ is a matroid if and only if it has no peerless antipode in any $(k,2k)$-hypersimplex for $k\geq 2$.
\end{corollary}

Furthermore, as a corollary of \Cref{prop:type+2+D}, we get a new characterisation in terms of local exchanges within $(2,4)$-hypersimplices and local `density' conditions within larger $(k,2k)$-hypersimplices.

\begin{corollary} \label{cor:type+2+A}
    Let $M$ be an equicardinal set system.
    Then $M$ is a matroid if and only if
    \begin{itemize}
        \item it has no peerless antipode in any $(2,4)$-hypersimplex, and
        \item it has no isolated antipode in any $(k,2k)$-hypersimplex for $k > 2$.
    \end{itemize}
\end{corollary}

\begin{remark}
    For those who prefer the set system perspective over the polyhedral, we give a rephrasing of \Cref{cor:type+2+A} as a basis exchange axiom.
    An equicardinal set system $M$ is a matroid if and only if it satisfies the following two properties:
    \begin{enumerate}
        \item for all $A,B \in M$ such that $|A\symDiff B| = 4$, there exist $a, b \in A\symDiff B$ such that $A \symDiff \{a,b\},\, B \symDiff \{a,b\} \in M$;
        \item for all $A,B \in M$ such that $|A\symDiff B| > 4$, there exists $C \in M\setminus\{A,B\}$ such that $A \cap B \subseteq C \subseteq A \cup B$.
    \end{enumerate}
    Property~1 is a `local' basis exchange axiom:
    it is equivalent to (and could be restated as) any other formulation of matroid basis exchange restricted to bases with $|A\symDiff B|=4$.
    In its presence, property~2 ensures that our set system is sufficiently dense that the local exchange axiom implies a global exchange axiom.
\end{remark}

\begin{remark} \label{rem:strong-vs-weak}
    Corollaries \ref{cor:type+1+A} and \ref{cor:type+2+A} are reminiscent of the difference between strong and weak matroids within the Baker--Bowler--Lorscheid framework of matroids over partial hyperstructures \cite{BakerBowler}.
    Explicitly, \Cref{cor:type+1+A} is equivalent to $M$ satisfying the Grassmann--Pl\"ucker relations over the Krasner hyperfield $\KK$, or being a strong $\KK$-matroid.
    In contrast, a weak matroid satisfies only the three-term Grassmann--Pl\"ucker relations, which over $\KK$ is equivalent to having no peerless antipode in any $(2,4)$-hypersimplex.
    Weak matroids also have the condition that their support must be a matroid, but over $\KK$ this is tautological.
    \Cref{cor:type+2+A} shows we can instead replace this with a local density condition on the support of the weak matroid.
\end{remark}

\subsection{Antipode cryptomorphisms for strong $\Delta$-matroids}

In the rest of this section, we will give two different characterisations of strong $\Delta$-matroids in terms of peerless and isolated antipodes.
As with even $\Delta$-matroids in \Cref{ssec:even-antipodes}, 
2-cubes and smaller pose no problems. 
The first characterisation bans peerless antipodes in all larger hypercubes (\Cref{thm:strong+iff+no+peerless}), and the second characterisation bans peerless antipodes in hypercubes of small dimension $>2$, and isolated antipodes in hypercubes of dimensions beyond (\Cref{prop:type+2+B}).

\begin{theorem}\label{thm:strong+iff+no+peerless}
    Let $M$ be a set system.
    Then $M$ is a strong $\Delta$-matroid if and only if it has no peerless antipode in any $k$-cube for $k\geq 3$.
\end{theorem}

To prove this, we will show that set systems with no peerless antipodes in $k$-cubes can be lifted to even $\Delta$-matroids.
To analyse this lift, we require the following terminology.

\begin{definition}
Given a cube $\cube{I}{J}$, we define its \defn{even demicube} $\evencube{I}{J}$ and \defn{odd demicube} $\oddcube{I}{J}$ respectively as
\begin{align*}
    \evencube{I}{J} &= \set{S \in \cube{I}{J}}{|S| \equiv 0 \bmod 2} \, ,  & \oddcube{I}{J} &= \set{S \in \cube{I}{J}}{|S| \equiv 1 \bmod 2} \, .
\end{align*}
A polytope $D \subseteq \RR^n$ is a \defn{$k$-demicube} if $D = \evencube{I}{J}$ or $D = \oddcube{I}{J}$ for some $I,J \subseteq [n]$ where $k = |I \symDiff J|$.
Observe that a $k$-demicube is $k$-dimensional for $k\ge3$.
\end{definition}
A trivial but useful observation is that the vertices of $\cube{I}{J}$ are partitioned into the odd and even demicubes $\oddcube{I}{J}$ and $\evencube{I}{J}$, as highlighted in the following example.

\begin{example}
    \label{ex:demi-cube-34}
    Suppose that $\cube{I}{J}$ is a $3$-dimensional cube or a $4$-dimensional cube.
    The demicubes of $\cube{I}{J}$ can be seen as a partition of the vertices into sets of like parity.
    If we suppose that $| I |$ is odd then all the black points below are in $\oddcube{I}{J}$ and all the white points are in $\evencube{I}{J}$.
    On the left we have the case of a $3$-dimensional cube and on the right the $4$-dimensional cube.
    
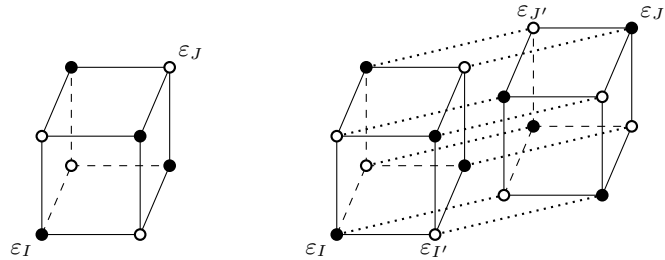
\begin{center}
        \begin{tikzpicture}
            [
                scale=1.3,
                vertex/.style={inner sep=1.4pt,circle,draw=black,fill=black,thick},
                wvertex/.style={inner sep=1.4pt,circle,draw=black,fill=white,thick},
            ]
            \begin{scope}
                \coordinate (bl1) at (0,0);
                \coordinate (br1) at (1,0);
                \coordinate (tl1) at (0,1);
                \coordinate (tr1) at (1,1);
    
                \coordinate (bl2) at (0.3,0.7);
                \coordinate (br2) at (1.3,0.7);
                \coordinate (tl2) at (0.3,1.7);
                \coordinate (tr2) at (1.3,1.7);
                
                \draw (bl1) -- (br1);
                \draw (bl1) -- (tl1);
                \draw (br1) -- (tr1);
                \draw (tl1) -- (tr1);
    
                \draw[dashed] (bl2) -- (br2);
                \draw[dashed] (bl2) -- (tl2);
                \draw (br2) -- (tr2);
                \draw (tl2) -- (tr2);
    
                \draw[dashed] (bl1) -- (bl2);
                \draw (tl1) -- (tl2);
                \draw (br1) -- (br2);
                \draw (tr1) -- (tr2);

                \node[below left] at (bl1) {$\cubeVertex{I}$};
                \node[above right] at (tr2) {$\cubeVertex{J}$};
                \node[vertex] at (bl1) {};
                \node[wvertex] at (tl1) {};
                \node[wvertex] at (br1) {};
                \node[vertex] at (tr1) {};
                
                \node[wvertex] at (tr2) {};
                \node[vertex] at (tl2) {};
                \node[wvertex] at (bl2) {};
                \node[vertex] at (br2) {};
            \end{scope}
            \begin{scope}[shift={(3, 0)}]
                \coordinate (bl1) at (0,0);
                \coordinate (br1) at (1,0);
                \coordinate (tl1) at (0,1);
                \coordinate (tr1) at (1,1);
    
                \coordinate (bl2) at (0.3,0.7);
                \coordinate (br2) at (1.3,0.7);
                \coordinate (tl2) at (0.3,1.7);
                \coordinate (tr2) at (1.3,1.7);

                \coordinate (bl12) at (1.7,0.4);
                \coordinate (br12) at (2.7,0.4);
                \coordinate (tl12) at (1.7,1.4);
                \coordinate (tr12) at (2.7,1.4);
    
                \coordinate (bl22) at (2.0,1.1);
                \coordinate (br22) at (3.0,1.1);
                \coordinate (tl22) at (2.0,2.1);
                \coordinate (tr22) at (3.0,2.1);
                
                \draw (bl1) -- (br1);
                \draw (bl1) -- (tl1);
                \draw (br1) -- (tr1);
                \draw (tl1) -- (tr1);
    
                \draw[dashed] (bl2) -- (br2);
                \draw[dashed] (bl2) -- (tl2);
                \draw (br2) -- (tr2);
                \draw (tl2) -- (tr2);
    
                \draw[dashed] (bl1) -- (bl2);
                \draw (tl1) -- (tl2);
                \draw (br1) -- (br2);
                \draw (tr1) -- (tr2);

                \draw (bl12) -- (br12);
                \draw (bl12) -- (tl12);
                \draw (br12) -- (tr12);
                \draw (tl12) -- (tr12);
    
                \draw[dashed] (bl22) -- (br22);
                \draw[dashed] (bl22) -- (tl22);
                \draw (br22) -- (tr22);
                \draw (tl22) -- (tr22);
    
                \draw[dashed] (bl12) -- (bl22);
                \draw (tl12) -- (tl22);
                \draw (br12) -- (br22);
                \draw (tr12) -- (tr22);

                \draw[dotted, thick] (bl1) -- (bl12);
                \draw[dotted, thick] (bl2) -- (bl22);
                \draw[dotted, thick] (br1) -- (br12);
                \draw[dotted, thick] (br2) -- (br22);
                \draw[dotted, thick] (tl1) -- (tl12);
                \draw[dotted, thick] (tl2) -- (tl22);
                \draw[dotted, thick] (tr1) -- (tr12);
                \draw[dotted, thick] (tr2) -- (tr22);
    
                \node[below left] at (bl1) {$\cubeVertex{I}$};
                \node[above right] at (tr2) {};% {$\cubeVertex{J}$};
                \node[vertex] at (bl1) {};
                \node[wvertex] at (tl1) {};
                \node[below] at (br1) {$\cubeVertex{I'}$};
                \node[wvertex] at (br1) {};
                \node[vertex] at (tr1) {};
                
                \node[wvertex] at (tr2) {};
                \node[vertex] at (tl2) {};
                \node[wvertex] at (bl2) {};
                \node[vertex] at (br2) {};

                \node[below left] at (bl12) {};% {$\cubeVertex{I}$};
                \node[above right] at (tr22) {$\cubeVertex{J}$};
                \node[wvertex] at (bl12) {};
                \node[vertex] at (tl12) {};
                \node[vertex] at (br12) {};
                \node[wvertex] at (tr12) {};
                
                \node[vertex] at (tr22) {};
                \node[above] at (tl22) {$\cubeVertex{J'}$};
                \node[wvertex] at (tl22) {};
                \node[vertex] at (bl22) {};
                \node[wvertex] at (br22) {};

            \end{scope}
        \end{tikzpicture}
    \end{center}
\end{example}

\begin{lemma}\label{lem:demi+projection}
    For any $i \in I \symDiff J$, the projection map $\pi_{i}\colon \RR^{[n]} \rightarrow \RR^{[n] - i}$ gives
    \[
    \pi_i(\evencube{I}{J}) = \pi_i(\oddcube{I}{J}) = \cube{I - i}{J - i} \, ,
    \]
    and induces a bijection on the vertices. 
\end{lemma}
\begin{proof}
    Because projection commutes with convex hull,
    it suffices to show the bijection on the vertices.
    We prove it for $\evencube{I}{J}$; the $\oddcube{I}{J}$ claim is very similar.
    First note that $\pi_i$ is injective, as each vertex $\cubeVertex{S} \in \cube{I-i}{J-i}$ has a unique set in its preimage: $\cubeVertex{S}$ if $S$ has even cardinality and $\cubeVertex{S + i}$ if $S$ has odd cardinality.
    Moreover, it is surjective as both polytopes have exactly $2^{|I\symDiff J| -1}$ vertices.
\end{proof}

\begin{example}
    We see this lemma in action by looking at the projection map $\pi_1$ for when $I = \varnothing$ and $J = \{1, 2, 3\}$.
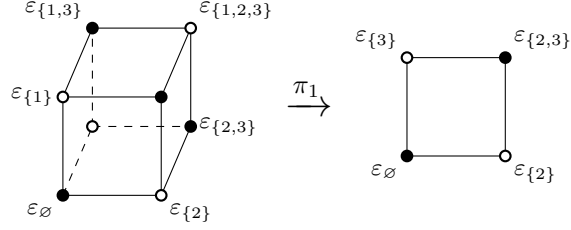
    \begin{center}
    \begin{tikzpicture}
            [
                scale=1.3,
                vertex/.style={inner sep=1.4pt,circle,draw=black,fill=black,thick},
                wvertex/.style={inner sep=1.4pt,circle,draw=black,fill=white,thick},
            ]
            \begin{scope}
                \coordinate (bl1) at (0,0);
                \coordinate (br1) at (1,0);
                \coordinate (tl1) at (0,1);
                \coordinate (tr1) at (1,1);
    
                \coordinate (bl2) at (0.3,0.7);
                \coordinate (br2) at (1.3,0.7);
                \coordinate (tl2) at (0.3,1.7);
                \coordinate (tr2) at (1.3,1.7);
                
                \draw (bl1) -- (br1);
                \draw (bl1) -- (tl1);
                \draw (br1) -- (tr1);
                \draw (tl1) -- (tr1);
    
                \draw[dashed] (bl2) -- (br2);
                \draw[dashed] (bl2) -- (tl2);
                \draw (br2) -- (tr2);
                \draw (tl2) -- (tr2);
    
                \draw[dashed] (bl1) -- (bl2);
                \draw (tl1) -- (tl2);
                \draw (br1) -- (br2);
                \draw (tr1) -- (tr2);

                \node[below left] at (bl1) {$\cubeVertex{\varnothing}$};
                \node[below right] at (br1)  {$\cubeVertex{\{2\}}$};
                \node[left] at (tl1)  {$\cubeVertex{\{1\}}$};

                \node[right] at (br2)  {$\cubeVertex{\{2, 3\}}$};
                \node[above left] at (tl2)  {$\cubeVertex{\{1, 3\}}$};
                \node[above right] at (tr2)  {$\cubeVertex{\{1, 2, 3\}}$};
                
                \node[vertex] at (bl1) {};
                \node[wvertex] at (tl1) {};
                \node[wvertex] at (br1) {};
                \node[vertex] at (tr1) {};
                
                \node[wvertex] at (tr2) {};
                \node[vertex] at (tl2) {};
                \node[wvertex] at (bl2) {};
                \node[vertex] at (br2) {};
            \end{scope}
            
            \begin{scope}[shift={(2.5, 1)}]
                \node at (0, 0) {{\Large $\xrightarrow{\pi_1}$}};
            \end{scope}
            
            \begin{scope}[shift={(3.5, 0.4)}]
                \coordinate (bl1) at (0,0);
                \coordinate (br1) at (1,0);
                \coordinate (tl1) at (0,1);
                \coordinate (tr1) at (1,1);
                
                \draw (bl1) -- (br1);
                \draw (bl1) -- (tl1);
                \draw (br1) -- (tr1);
                \draw (tl1) -- (tr1);
    
                \node[below left] at (bl1) {$\cubeVertex{\varnothing}$};
                \node[below right] at (br1)  {$\cubeVertex{\{2\}}$};
                \node[above left] at (tl1)  {$\cubeVertex{\{3\}}$};
                \node[above right] at (tr1)  {$\cubeVertex{\{2, 3\}}$};
                
                \node[vertex] at (bl1) {};
                \node[wvertex] at (tl1) {};
                \node[wvertex] at (br1) {};
                \node[vertex] at (tr1) {};
            \end{scope}
        \end{tikzpicture}
    \end{center}
    Note that two vertices in $\cube{\varnothing}{\{1, 2, 3\}}$ get mapped to each vertex in $\cube{\varnothing}{\{2, 3\}}$. Each of these pairs of vertices have a different parity as one contains $1$ and the other doesn't.
    For example
    \[
        \pi_1\left(\cubeVertex{\{1, 3\}}\right) = \pi_1\left(\cubeVertex{\{3\}}\right) = \cubeVertex{\{3\}}.
    \]
\end{example}

We now turn our attention to antipodes in demicubes.
Observe that any sets $S,T$ with $S\cap T = I \cap J$ and $S\symDiff T = I \symDiff J$ define the same demicubes $\evencube{S}{T} = \evencube{I}{J}$ and $\oddcube{S}{T} = \oddcube{I}{J}$.
However, the vertices $\cubeVertex{S}$ and $\cubeVertex{T}$ are only contained in $\evencube{I}{J}$ (resp.\ $\oddcube{I}{J}$) if they have the correct parity.
We say the pair $(\cubeVertex{S},\cubeVertex{T})$ is an \emph{antipode} of the demicube $\evencube{I}{J}$ (resp.\ $\oddcube{I}{J}$) if it is an antipode of the cube $\cube{I}{J}$ and contained in the demicube $\evencube{I}{J}$ (resp.\ $\oddcube{I}{J}$).
Note that for $(\cubeVertex{S},\cubeVertex{T})$ to be contained in a demicube, both sets must have the same parity and so we necessarily have $|S \symDiff T|$ is even.
It follows that $k$-demicubes have antipodes only when $k$ is even.

Given some demicube $D$, we say that $(\cubeVertex{I},\cubeVertex{J}) \subseteq P(M)$ is a \defn{peerless antipode in $D$} if it is the only antipode of $D$ contained in $P(M)$.
Similarly, we say that $(\cubeVertex{I},\cubeVertex{J}) \subseteq P(M)$ is an \defn{isolated antipode in $D$} if the restriction of $P(M)$ to $D$ is exactly $(\cubeVertex{I},\cubeVertex{J})$.
Note that in both cases, this is a strictly stronger condition than being a peerless or isolated antipode according to \Cref{def:peerless+isolated}. 
The demicube $D$ is a subpolytope of $\cube{I}{J}$
(with $|I\symDiff J|$ even),
so if $\cube{I}{J}$ contains exactly two antipodes, one of each parity,
then each antipode becomes peerless in its demicube.

\begin{lemma}\label{lem:peerless+cubes+imply+demicubes}
    Suppose $M$ has no peerless antipodes in any $k$-cube for $k =3,4$.
    Then $M$ also has no peerless antipodes in any $4$-demicube.
\end{lemma}
\begin{proof}
    For ease of notation, we let $\cQ_{X, Y} = \cube{X}{Y} \cap \polyOf{\matroid}$ and invite the reader to use \Cref{ex:demi-cube-34} when reading this proof.
    Suppose for contradiction that $(\cubeVertex{I}, \cubeVertex{J}) \subseteq \polyOf{\matroid}$ is a peerless antipode in a $4$-demicube $D$.
    Let $D'$ be the other $4$-demicube of $\cube{I}{J}$.
    By a suitable relabelling, we may assume that $I = \varnothing$ and $J = \{1, 2, 3, 4\}$, i.e.\ $D = \evencube{\emptyset}{[4]}$ and $D' = \oddcube{\emptyset}{[4]}$.
    %and let $\cD = D \cap \polyOf{\matroid}$. 
    %Let $\cQ = \cQ_{I, J}$ be a $4$-cube of $\polyOf{\matroid}$ which contains as a demicube and let $\cD'$ be the other $4$-demicube of $\polyOf{\matroid}$.

    Let $\cQ = \cQ_{I, J}$, and define $\cD = D \cap \polyOf{\matroid}$ and $\cD' = D' \cap \polyOf{\matroid}$.
    As $\cQ$ has no peerless antipodes by assumption, $\cD'$ must contain an antipode, say $(\cubeVertex{I'}, \cubeVertex{J'})$.
    All antipodes of $\cD'$ are of the form $I' = \{i\}$ and $J' = [n] - \{i\}$,
    hence up to another relabelling we may assume that $I' = \{1\}$ and $J' = \{2, 3, 4\}$.
    %Without loss of generality we may assume that $(\cubeVertex{I}, \cubeVertex{I'})$ and $(\cubeVertex{J}, \cubeVertex{J'})$ are edges such that $\cube{I}{I'}$ and $\cube{J}{J'}$ are $2$-cubes where $I \symDiff I' = J \symDiff J'$.
    %As before, without loss of generality, up to a relabelling we may assume that $I' = \{1\} = J\symDiff J'$ implying $J' = \{2, 3, 4\}$.
    It follows that $\cube{I}{J'}$ and $\cube{I'}{J}$ are $3$-cubes with antipodes $(\cubeVertex{I}, \cubeVertex{J'})$ and $(\cubeVertex{I'}, \cubeVertex{J})$ respectively that are both present in $\polyOf{\matroid}$. 
    By assumption, $(\cubeVertex{I}, \cubeVertex{J'})$ and $(\cubeVertex{I'}, \cubeVertex{J})$ must not be peerless antipodes in $\cQ_{I, J'}$ and $\cQ_{I', J}$ respectively, and therefore each of $\cQ_{I, J'}$ and $\cQ_{I', J}$ must contain an additional antipode, whose vertices have different parity.
    In other words there is an antipode $(\cubeVertex{A}, \cubeVertex{A'})$ in $\cQ_{I, J'}$ and an antipode $(\cubeVertex{B}, \cubeVertex{B'})$ in $\cQ_{I', J}$ such that $\cubeVertex{A}, \cubeVertex{B} \in \cD$ and $\cubeVertex{A'}$, $\cubeVertex{B'} \in \cD'$ (up to relabelling).
    There are three potential antipodes in each of $\cQ_{I, J'}$ and $\cQ_{I', J}$, respectively:
    \begin{align}
        A &= \{3, 4\} \text{ and } A' = \{2\} && B = \{1, 2\} \text{ and } B' = \{1, 3, 4\}\nonumber\\
        A &= \{2, 4\} \text{ and } A' = \{3\} && B = \{1, 3\} \text{ and } B' = \{1, 2, 4\}\label{ex:optionsAB}\\
        A &= \{2, 3\} \text{ and } A' = \{4\} && B = \{1, 4\} \text{ and } B' = \{1, 2, 3\}\nonumber
    \end{align}
    
    In each case $I$, $J$, $A$, and $B$ are all distinct.
    As $(\cubeVertex{I}, \cubeVertex{J})$ is a peerless antipode in $\cD$, we know that $(\cubeVertex{A}, \cubeVertex{B})$ is not an antipode in $\cQ$.
    %If $(\cubeVertex{A}, \cubeVertex{B})$ is an antipode in $\cQ$ (and therefore in $\cD$), this contradicts our assumption that $(\cubeVertex{I}, \cubeVertex{J})$ is a peerless antipode in $\cD$; therefore it is not an antipode in $\cQ$.
    In other words, $A$ and $B$ cannot be in the same row in (\ref{ex:optionsAB}).
    We have enough symmetry left to permute the labels $\{2,3,4\}$ arbitrarily, and this acts transitively on the pairs of $A$ and~$B$ not in the same row.
    So we relabel $\{2,3,4\}$ to assume
    \[
        A = \{3, 4\},\, A' = \{2\},\,B = \{1, 3\},\, B' = \{1, 2, 4\}.
    \]
    Recall also that
    \[
        I = \varnothing,\, I' = \{1\},\, J = \{1, 2, 3, 4\},\, J' = \{2, 3, 4\}\,.
    \]
    We study the $3$-cubes $\cube{I}{B'}$ and $\cube{A'}{J}$.%, $\cube{I'}{A}$ and $\cube{B}{J'}$.
    
    Focusing first on $\cube{I}{B'}$, the only points from the set $\{I', J, J', A, A', B\}$ contained in $\cube{I}{B'}$ are $A'$ and $I'$.
    Furthermore, as $\cQ_{I, B'}$ has no peerless antipodes, there must exist some point $\cubeVertex{C} \in \cD$ to make an antipode in $\cQ_{I, B'}$ where $C$ is distinct from $A$, $B$, $I$, and $J$.
    Either $C = \{1, 4\}$ with $A'$ as an antipode, $C = \{2, 4\}$ with $I'$ as its antipode or $C = \{1, 2\}$ and its antipode $\{4\}$ must also be present.
    If $C = \{2, 4\}$ then $(\cubeVertex{C}, \cubeVertex{B})$ is an antipode in  $\cQ$ and therefore in $\cD$, contradicting our assumption.
    Similarly, if $C = \{1, 2\}$ then $(\cubeVertex{C}, \cubeVertex{A})$ is an antipode in $\cQ$.
    Therefore $C = \{1, 4\}$.

    By a similar process as above for $\cube{A'}{J}$, only the points $B'$ and $J'$ are contained in $\cube{A'}{J}$. 
    Since $\cQ_{A', J}$ has no peerless antipode, there must exist some point $\cubeVertex{D} \in \cD$ to make an additional antipode in $\cQ_{A', J}$.
    In this case, either $D = \{1, 2\}$ with $J'$ as its antipode, $D = \{2, 3\}$ with $B'$ as its antipode or $D = \{2, 4\}$ and its antipode $\{1, 2, 3\}$ must also be present.
    As happened above, if $D = \{1, 2\}$ or $D = \{2, 4\}$ we get a contradiction.
    Therefore, $D = \{2, 3\}$.

    But then $C$ and $D$ are present in $\cQ$ and are antipodes,
    contradicting the fact that $(\cubeVertex{I}, \cubeVertex{J})$ is the only antipode in $\cQ$ and therefore in $\cD$.
    Since all of our cases have ended in contradiction, our initial assumption that $\cD$ is a $4$-demicube of $\polyOf{\matroid}$ with a peerless antipode is false, thus proving the lemma.
\end{proof}

The following lemma shows that \eqref{it:4-cubes} implies \eqref{it:3-even-Delta} in \Cref{thm:A}.

\begin{lemma}\label{lem:peerless+implies+lift+even+delta}
    Suppose $M$ has no peerless antipodes in any $k$-cube for $k \geq 3$.
    Then $\overline{M}$ is an even $\Delta$-matroid.
\end{lemma}
\begin{proof}
    We show $\overline{M}$ is an even $\Delta$-matroid via the characterisation in \Cref{prop:type+2+D}: that $\overline{M}$ has no peerless antipode in any $4$-cube and no isolated antipode in any $2k$-cube.
    
    We first show that $\overline{M}$ has no peerless antipode in any $4$-cube $\cube{\overline{I}}{\overline{J}}$ where $|\overline{I} \symDiff \overline{J}| = 4$.
    Note that the restriction $\polyOf{\overline{M}} \cap \cube{\overline{I}}{\overline{J}}$ to a cube is equal to the restriction $\polyOf{\overline{M}} \cap \evencube{\overline{I}}{\overline{J}}$, as all sets of $\overline{M}$ are of even cardinality.
    %First consider the case where $|\overline{I} \symDiff \overline{J}| = 4$ and $* \notin \overline{I} \symDiff \overline{J}$.
    We proceed via a case check depending on whether $*$ is contained in $\overline{I}, \overline{J}$, both or neither.
    If $* \in \overline{I} \cap \overline{J}$, then $\evencube{\overline{I}}{\overline{J}}$ is the translated 4-demicube $\oddcube{I}{J} + \cubeVertex{*}$.
    As $\oddcube{I}{J}$ is a 4-demicube, it follows from \Cref{lem:peerless+cubes+imply+demicubes} that $\polyOf{M} \cap \oddcube{I}{J}$ contains no peerless antipodes, hence the same holds true for $\polyOf{\overline{M}} \cap \evencube{\overline{I}}{\overline{J}}$.
    If $* \notin \overline{I} \cup \overline{J}$, then $\evencube{\overline{I}}{\overline{J}} = \evencube{I}{J}$, and so it follows from \Cref{lem:peerless+cubes+imply+demicubes} that $\polyOf{M} \cap \evencube{I}{J} = \polyOf{\overline{M}} \cap \evencube{\overline{I}}{\overline{J}}$ contains no peerless antipodes.
    Finally if $* \in \overline{I} \symDiff \overline{J}$, then Lemma~\ref{lem:demi+projection} implies that $\pi_*(\evencube{\overline{I}}{\overline{J}}) = \cube{I}{J}$ is a bijection on the vertices, easily seen to preserve antipodality.
    Because $\polyOf{M} \cap \cube{I}{J}$ contains no peerless antipodes, we get that $\polyOf{\overline{M}} \cap \evencube{\overline{I}}{\overline{J}}$ contains no peerless antipodes either.

    We next show that $\overline{M}$ has no isolated antipode in any $2k$-cube $\cube{\overline{I}}{\overline{J}}$ where $|\overline{I} \symDiff \overline{J}| = 2k$ for $k > 2$.
    This is again broken into cases depending on whether $*$ is contained in $\overline{I}, \overline{J}$, both or neither.
    If $* \in \overline{I} \symDiff \overline{J}$, then the proof is identical to the $k=2$ case and $\polyOf{\overline{M}} \cap \evencube{\overline{I}}{\overline{J}}$ contains no peerless antipodes, implying also no isolated antipodes.

    %We finally consider the case where $|\overline{I} \symDiff \overline{J}| \geq 2k$ for $k > 2$ and $* \notin \overline{I} \symDiff \overline{J}$.    
    Suppose $* \in \overline{I} \cap \overline{J}$, and that $(\cubeVertex{\overline{I}},\cubeVertex{\overline{J}})$ is an isolated antipode in $\polyOf{\overline{M}}$.
    This implies that both $I$ and~$J$ must be odd cardinality sets, and that $\polyOf{M} \cap \cube{I}{J}$ must contain no other $\cubeVertex{K}$ where $K$ is a set with odd cardinality.
    As such, the vertices of $\polyOf{M} \cap \cube{I}{J}$ are either $\cubeVertex{I}$ or $\cubeVertex{J}$ or contained in $\evencube{I}{J}$.
    As $\polyOf{M}$ has no peerless antipodes, there exists some other antipode $(\cubeVertex{I'},\cubeVertex{J'}) \in \polyOf{M} \cap \cube{I}{J}$ where $I'$ and $J'$ have even cardinality.
    Note that
    \[
    |I \symDiff I'| + |I \symDiff J'| = |I \symDiff J| = 2k > 4
    \]
    and so without loss of generality $|I \symDiff J'| \geq k \geq 3$.
    Then $(\cubeVertex{I}, \cubeVertex{J'})$ is an antipode of $\polyOf{M} \cap \cube{I}{J'}$ that also cannot be peerless, so there exists another antipode $(\cubeVertex{I''}, \cubeVertex{J''}) \in \polyOf{M} \cap \cube{I}{J'}$ with $I'', J''$ distinct from $I$ and $J'$.
    However, $I$ and $J'$ have opposite parity, and so without loss of generality $I''$ has odd cardinality.
    Moreover, $\cubeVertex{I''}$ is a vertex of $\polyOf{M} \cap \cube{I}{J}$ distinct from $\cubeVertex{I}$ and $\cubeVertex{J}$, which contradicts that $(\cubeVertex{\overline{I}},\cubeVertex{\overline{J}})$ is isolated.
    The case where $* \notin \overline{I} \cap \overline{J}$ is very similar; we just reverse the parity in the previous argument.
    %As $\overline{M}$ has no peerless antipodes in $4$-cubes and no isolated antipodes in any $2k$-cubes, it is an even $\Delta$-matroid by Proposition~\ref{prop:type+2+D}.
    %\aramTodo{Where do we see that there are no isolated antipodes in any $2k$-cube?}
\end{proof}

\begin{proof}[Proof of Theorem~\ref{thm:strong+iff+no+peerless}]
If $M$ is a set system with no peerless antipodes in $k$-cubes for $k \geq 3$, then \Cref{lem:peerless+implies+lift+even+delta} gives that $\overline{M}$ is an even $\Delta$-matroid.
By \Cref{thm:hyperplane-equals-wenzel}, this is equivalent to $M$ being a strong $\Delta$-matroid.

Conversely, let $M$ be a strong $\Delta$-matroid and consider a peerless antipode $(\cubeVertex{I}, \cubeVertex{J})$.
By the hyperplane exchange property, there exists $i,j \in I \symDiff J$ such that $I' = I \symDiff \{i,j\}$ and $J' = J \symDiff \{i,j\}$ are in $M$.
In particular, $(\cubeVertex{I'}, \cubeVertex{J'})$ is an antipode in $\polyOf{M} \cap \cube{I}{J}$.
However, $(\cubeVertex{I}, \cubeVertex{J})$ being peerless implies that $I = J'$ and $J = I'$.
In particular, we have that $|I \symDiff J| \geq 3$, and so $\Delta$ has no peerless antipodes in $k$-cubes for $k \geq 3$.
\end{proof}

As with even $\Delta$-matroids, we can restrict the peerless condition to just the low-dimensional cubes, and insist on the much weaker isolated antipode condition for higher dimensional cubes.
This gives the following equivalent characterisation of strong $\Delta$-matroids.

\begin{proposition}\label{prop:type+2+B}
    Let $M$ be a set system.
    Then $M$ is a strong $\Delta$-matroid if and only if
    \begin{itemize}
        \item it has no peerless antipode in any $3$-cube or $4$-cube, and
        \item it has no isolated antipode in any $k$-cube for $k \geq 5$.
    \end{itemize}
\end{proposition}
\begin{proof}
    By Theorem~\ref{thm:strong+iff+no+peerless}, any strong $\Delta$-matroid automatically satisfies the claim as no peerless antipodes implies no isolated antipodes.
    
    Conversely, assume $M$ is a set system with no peerless antipode in any $3,4$-cube or isolated antipode in any $k$-cube for $k \geq 5$.
    We show by induction on $k$ that $M$ has no peerless antipode in any $k$-cube for $k \geq 3$.
    The base cases on $k=3,4$ follow from assumption.
    Suppose this hypothesis holds for all $3 \leq k' < k$, and for contradiction that there exists some peerless antipode in a $k$-cube.
    Without loss of generality, we may assume this $k$-cube is $\cube{\emptyset}{[k]}$.
    Let $M_k = \set{A \subseteq [k]}{A \in M}$ be the restriction of $M$ to this $k$-cube: by the induction hypothesis, $M_k$ has no peerless antipodes in any $k'$-cube for $k' < k$.
    As $M_k$ has a peerless antipode, it is not a strong $\Delta$-matroid by \Cref{thm:strong+iff+no+peerless} and hence $\overline{M_k}$ is not an even $\Delta$-matroid by \Cref{thm:A}.
    We show that $\overline{M_k}$ has no peerless antipode in any $4$-cube and no isolated antipode in any $2\ell$-cube for $\ell > 2$: as these conditions characterise even $\Delta$-matroids (\Cref{prop:type+2+D}), this gives a contradiction.

    The proof that $\overline{M_k}$ has no peerless antipode in any $4$-cube is identical to the proof of the same claim in \Cref{lem:peerless+implies+lift+even+delta}.
    It remains to show that $\overline{M_k}$ has no isolated antipode in any $2\ell$-cube $\cube{\overline{I}}{\overline{J}}$ where $|\overline{I} \symDiff \overline{J}| = 2\ell$ for $\ell > 2$.
    If $* \in \overline{I} \symDiff \overline{J}$, this is likewise a claim that has been established in the proof of \Cref{lem:peerless+implies+lift+even+delta}.

    Suppose $* \in \overline{I} \cap \overline{J}$, and that $(\cubeVertex{\overline{I}},\cubeVertex{\overline{J}})$ is an isolated antipode in $\polyOf{\overline{M_k}}$.
    This implies that both $I$ and~$J$ must be odd cardinality sets, and that $\polyOf{M_k} \cap \cube{I}{J}$ must contain no other $\cubeVertex{L}$ where $L$ is a set with odd cardinality.
    As such, the vertices of $\polyOf{M_k} \cap \cube{I}{J}$ are either $\cubeVertex{I}$ or $\cubeVertex{J}$ or contained in $\evencube{I}{J}$.
    As $\polyOf{M_k}$ has no isolated antipodes by assumption, there exists some other point $\cubeVertex{K} \in \polyOf{M_k} \cap \cube{I}{J}$ where $K$ has even cardinality.
    Note that
    \[
    |I \symDiff K| + |J \symDiff K| = |I \symDiff J| = 2\ell > 4
    \]
    and so without loss of generality $|I \symDiff K| \geq \ell \geq 3$.
    Note also that $|I \symDiff K| < 2\ell = k$, and so by the induction hypothesis $(\cubeVertex{I}, \cubeVertex{K})$ is not a peerless antipode and there exists some $(\cubeVertex{I'}, \cubeVertex{K'}) \in \polyOf{M_k} \cap \cube{I}{K}$ with $I', K'$ distinct from $I, K$.
    However, $I$ and $K$ have opposite parity, and so without loss of generality $I'$ has odd cardinality.
    Moreover, $\cubeVertex{I'}$ is a vertex of $\polyOf{M_k} \cap \cube{I}{J}$ distinct from $\cubeVertex{I}$ and $\cubeVertex{J}$, which contradicts that $(\cubeVertex{\overline{I}},\cubeVertex{\overline{J}})$ is isolated.
    The case where $* \notin \overline{I} \cap \overline{J}$ is very similar; we just reverse the parity in the previous argument.

    In all cases, $\overline{M_k}$ satisfies the properties of an even $\Delta$-matroid, leading to a contradiction.
    It follows that $M$ has no peerless antipode in any $k$-cube for $k \geq 3$.
    %
    %
    %By \Cref{prop:antipode-implies-delta-matroid-2}, this implies that $M$ is a $\Delta$-matroid.
    %, but is not a strong $\Delta$-matroid.
    % Then there exists some edge $(\cubeVertex{I}, \cubeVertex{J})$ with $|I \symDiff J| = k$ for $k \geq 3$; assume it is an edge of minimal length with this condition.
    % Lemma~\ref{lem:edge-implies-peerless} immediately rules out the case that $k=3, 4$, and so we can assume $k>4$.
    % As $(\cubeVertex{I}, \cubeVertex{J})$ is not isolated, there exists a two-dimensional face $F \subseteq \cube{I}{J}$ of $\polyOf{M}$ that contains $(\cubeVertex{I}, \cubeVertex{J})$ as an edge.
    % As $(\cubeVertex{I}, \cubeVertex{J})$ is peerless by Lemma~\ref{lem:edge-implies-peerless}, the other edges in $F$ are necessarily shorter, and so by our assumption all parallel to some $e_i$ or $e_i \pm e_j$.
    % However, it takes at least $\ceil{ k/2 }$ affinely independent edges of the form $e_i$ and $e_i \pm e_j$ to walk from $\cubeVertex{I}$ to $\cubeVertex{J}$, implying that $\dim(F) \geq \ceil{ k/2} > 2$, giving a contradiction.
\end{proof}

\section{Peerless antipodes from embedding equations}\label{s:peerlessantipodegrassmann}

In this section, we show how the peerless antipode characterisation in \Cref{thm:strong+iff+no+peerless} can be viewed as a `combinatorialisation' of equations cutting out the orthogonal Grassmannian $\mathrm{OGr}(n,2n+1)$.
To do so, we will recap a number of ideas and results from \cite{CDFS1}.
This recap will be light touch to hide some (if not all) of the ugly technicalities: we encourage the interested reader to go to the original paper for further details. 

We briefly emphasise how our present work differs from \cite{CDFS1}.
In this section we use a basis of the equations cutting out the orthogonal Grassmannian, calculated in \cite[Section A.2.1]{CDFS1} from the results of \cite{Lichtenstein:1982}.
In \cite{CDFS1} we take this basis and produce a spanning set of equations with small support \cite[Theorem 4.8 and Section A.2.2]{CDFS1} to tropicalise, whilst in this section we tropicalise the large support basis.
This dichotomy is analogous to the one in the type~A setting between the straightening relations (basis) and the Pl\"{u}cker equations (minimal support) for the Grassmannian.

\subsection{Peerless antipode equations}

We state our combinatorial equations in the language of tropical geometry.
For context and a detailed treatment of tropical algebra and geometry, see \cite{MaclaganSturmfels}.

Let $\BB=(\{0,1\},\oplus,\odot)$ be the Boolean semifield,
where $0$ is the zero, $1$ the multiplicative identity, and $1\oplus 1=1$.
For those familiar with the theory of matroids over partial hyperstructures \cite{BakerBowler,JinKimO,KimSp}, we could instead set up our equations over the Krasner hyperfield.

Our equations are in the finite set of indeterminates $X=\set{x_S}{S\subseteq[n]}$.
Define $\BB[X]$ to be the 
%monoid semiring with coefficients in~$\BB$ of the free commutative monoid $\mathbb N^X=\Hom_{\mathrm{Set}}(X,\mathbb N)$ on the generators $X$.
%Explicitly, for $a\in\mathbb N^X$, let $X^{\odot a}$ denote the formal monomial 
%\raisebox{0pt}[0pt][0pt]{$\bigodot_{x\in X}\underbrace{x\odot\cdots\odot x}_{a(x)\mbox{\scriptsize\ times}}$}. % \raisebox so that the underbrace can sit in the space of the displayed equation
set of formal sums
\[\bigoplus_{a\in A}X^{\odot a}\]
for some finite subset $A\subset\mathbb N^X$,
where $X^{\odot a}$ is the formal monomial 
\raisebox{0pt}[0pt][0pt]{$\bigodot_{x_S\in X}\overbrace{x_S\odot\cdots\odot x_S}^{a(x_S)\mbox{\scriptsize\ times}}$}. % \raisebox so that the underbrace can sit in the space of the displayed equation
%Addition in $\BB[X]$ corresponds to union of the sets~$A$, and multiplication to Minkowski sum.
In fact, the elements of~$\BB[X]$ we use in this paper are all square-free, that is, $a(x_S)\in\{0,1\}$.
(In \cite{CDFS1} we presented a semiring structure on~$\BB[X]$.)

When we call an element $f\in\BB[X]$ a \defn{tropical equation}, 
we are thinking of it as a condition that may or may not be satisfied by a tuple $p\in\BB^X$.
The tropical equation $f=\bigoplus_{a\in A}X^{\odot a}$ is \defn{satisfied} by~$p$ 
if the number of $a\in A$ such that $p^{\odot a}=1$ -- i.e.\ such that $p(x_S)=1$ for all $x_S$ with $a(x_S)>0$ -- 
is not exactly~1.

The \defn{tropicalisation} map $\trop:\CC[X]\to\BB[X]$
is defined by $\trop(\sum_{a\in A} z_a X^a) = \bigoplus_{a\in A} X^{\odot a}$ 
if $z_a\in\CC$ is nonzero for all $a\in A$.
We warn the reader that tropicalisation is not a homomorphism of additive semigroups, on account of cancellations in~$\CC[X]$.
But tropicalisation is compatible with the property of satisfying an equation, in the following sense.
Suppose $F\in\CC[X]$ is a complex polynomial such that $F(P)=0$ at the point $P\in\CC^X$.
Let $f=\trop(F)$, and let $p\in\BB^X$ be the support of~$P$, i.e.\ $p(x_S)=1$ if and only if $P(x_S)\ne0$.
Then the tropical equation $f$ is satisfied by~$p$.

%When we present tropical exchange equations later in this section,
%these will come with lower bounds on the sizes of sets involved.
%Comparing to the equations over $\CC[X]$ in \Cref{s:eqsfromGrass}\aram{ref}\alex{This sentence refers to a section of \cite{CDFS1} which was not copied over.}, we see that these bounds arise from ``sporadic'' cancellations of terms when the sets are small, making the tropical equation not agree with the tropicalisation.

Set systems $M$ on~$[n]$ are in bijection with tuples $\nu_M\in\BB^X$
by the rule that $\nu_M(x_S)=1$ if and only if $S\in M$.
In this way we may make sense of saying that a set system $M$ satisfies a tropical equation $f\in\BB[X]$:
we mean that $\nu_M$ satisfies~$f$.

To each $I, J$ with $|I \symDiff J| \geq 3$, we can associate a tropical equation obtained by summing over the antipodes of the cube $\cube{I}{J}$:
\begin{definition} \label{defn:peerless+antipode+equations}
    The \defn{peerless antipode equations} on $[n]$ are the the collection of tropical equations 
    \begin{align} \label{eq:peerless+antipode}
        \cG_n \coloneqq& \set{g_{I,J}}{I, J \subseteq [n],\, |I\symDiff J| \ge 3} \nonumber \\
        \text{where} \quad g_{I,J} =& \bigoplus_{\substack{(\cubeVertex{S},\cubeVertex{T}) \\\text{ antipodes in } \cube{I}{J}}} x_{S} \odot x_{T} % \, , \quad \cA(\cube{I}{J}) = \{(S,T) \st (\cubeVertex{S},\cubeVertex{T}) \text{ antipode of } \cube{I}{J}\} \\
    \end{align} 
\end{definition}
In the same way that each cube can be expressed as $\cube{I}{J}$ for any pair of antipodes $(\cubeVertex{I},\cubeVertex{J})$, the equations $g_{I,J}$ and $g_{S,T}$ are equal when $I \cap J = S \cap T$ and $I \cup J = S \cup T$.
In fact, any pair of peerless antipode equations are either equal or have disjoint monomial supports.
We consider $\cG_n$ as a set rather than a multiset, and so each equation only appears once.

\begin{corollary}
    Let $M$ be a set system on $[n]$.
    Then $M$ is a strong $\Delta$-matroid if and only if $\nu_M$ satisfies the peerless antipode equations $\cG_n$.
\end{corollary}
\begin{proof}
    By \Cref{thm:strong+iff+no+peerless}, $M$ is not a strong $\Delta$-matroid if and only if there exists some $\cube{I}{J}$ such that $P(M) \cap \cube{I}{J}$ has a single antipode.
    Equivalently, there is a unique monomial in $g_{I,J}$ which equals $1$ when evaluated on $\nu_M$.
    This is precisely when $\nu_M$ does not satisfy $g_{I,J}$.
\end{proof}

\subsection{Quadrics defining $\mathrm{OGr}(n,2n+1)$}

The orthogonal Grassmannian $\mathrm{OGr}(n,2n+1)$ is the moduli space of isotropic subspaces of $\CC^n$ of real dimension $n$.
It can be considered as a homogeneous space $\bbG /\bbP$ where $\bbG = \mathrm{O}(2n+1)$ is the complex orthogonal group and $\bbP$ is the stabiliser of a fixed choice of maximal isotropic subspace.
Representation theory gives a uniform procedure for defining an embedding of $\bbG/\bbP$ into a projectivised $\bbG$-representation $\proj(V_\lambda)$, where $V_\lambda$ is related to the datum for the parabolic $\bbP$.
Furthermore, the ideal cutting out the embedding is generated by quadrics and there is a recipe for how to obtain these quadrics~\cite[Theorem~1]{Lichtenstein:1982}.
For us, this $\bbG$-representation is the \emph{spin representation} $\Sp$, a $2^n$-dimensional representation whose weights are in bijection with vertices of the $n$-cube.
As such, we can consider $\mathrm{OGr}(n,2n+1) \subset \PP^{2^n -1}$ as a projective variety whose defining ideal $\cI_n$ is generated by some quadrics in the polynomial ring $\CC[x_I \colon I \subseteq [n]]$.

In \cite{CDFS1}, the authors gave an explicit description of $\cI_n$ and the quadrics generating it, which we now recall with some simplifications.
To state these equations, we introduce a different indexing for hypercubes.
Given a set $I$, write $I^c = [n] \setminus I$ for its complement.
For disjoint $N, L \subseteq [n]$, define the cube
\[
Q_{N,L} = \conv\set{\cubeVertex{K}}{N \subset K, L \cap K = \emptyset}  \, .
\]
This is a cube of dimension $n - |N| - |L|$.
The antipodes of $Q_{N,L}$ are exactly the pairs of vertices of the form $(\cubeVertex{K}, \cubeVertex{\overline{K}})$ where $\overline{K} := (K^c \setminus L) \cup N$.
We can relate $Q_{N,L}$ to our previous cube notation by observing $\cube{I}{J} = Q_{I \cap J, (I \cup J)^c}$.
However, unlike $\cube{I}{J}$ which is equal to $\cube{I'}{J'}$ for any antipode $(I',J')$ in $\cube{I}{J}$, 
the cube $Q_{N,L}$ uniquely determines $N$ and~$L$.
% This is the set of subsets indexing a face of the $n$-cube $Q_n$; we'll generally speak of it as a face or a (sub)cube itself.
% The full $n$-cube $Q_n$ is given by $Q_{\emptyset,\emptyset}$.

% %We now reindex the quadratic embedding equation $F_{I,J}^{(B)}$ in this notation.
% Any pair $(I,J)$ of subsets of~$[n]$ is an antipode in a unique cube, namely $Q_{N,L}$, where $N = I \cap J$ and $L = (I \cup J)^c$ (this may be a $0$-cube).
% %Every pair $(I,J)$ of subsets of~$[n]$ is an antipode of a unique cube, where we think of $(I,I)$ as an antipode of the $0$-cube $\{I\}$,
% So we will call such a pair of subsets an \defn{ordered antipode}.
% Observe that we can write $J = (I^c \setminus L) \cup N$.
% As such, for any $K \in Q_{N,L}$ we let $\overline{K}_{N,L} := (K^c \setminus L) \cup N$ denote the set such that $(K, \overline{K}_{N,L})$ is the unique antipode in $Q_{N,L}$ containing $K$.
% If $N$ and $L$ are clear from context we will drop them and simplify $\overline{K}_{N,L}$ to $\overline{K}$. Letting $S = (N \cup L)^c$, we have $\overline{K}_{N,L} = K \symDiff S$. 
Finally, we let $\cM_{N,L}$ denote the set
\begin{equation}\label{eq:M_NL}
    \cM_{N,L} \coloneqq \set{4k + 2 \ceil{\frac{n}{2}} - |N| - |L| - \delta}{k\in\ZZ \, , \, \delta \in \{0,1\}} \setminus \Big\{ m-2,\,m-1,\,m,\,m+1,\,m+2\Big\}
\end{equation}
where $m= n -|N|-|L|$.
With this, we define
\begin{align*}
    \cP_n &= \set{p^M_{N,L}}{M,N,L \subseteq [n] \text{ pairwise disjoint} \, , \, 2|M| \in \cM_{N,L}} \subseteq \CC[x_I : I \subseteq [n]] \\
    \text{ where } p^M_{N,L} &= \sum_{\substack{(\cubeVertex{K},\cubeVertex{\overline{K}}) \\\text{ antipodes in } Q_{N,L}}} s(K,M,N,L)x_{K}x_{{\overline{K}}} \, , 
    \end{align*}
    where $s(K,M,N,L)\in \{\pm 1\}$ is some sign depending on the choice of $K,M,N,L$: the precise definition is given in \cite[Theorem A.10]{CDFS1}, but the support of $p_{N,L}^M$ is sufficient for our purposes.
    By explicitly excluding $\{m-2,\ldots,m+2\}$ from $\cM_{N,L}$, it turns out that no equation in $\cP_n$ is supported on antipodes of $0$-, $1$- and $2$-cubes, though this is not a priori obvious.

%AF TODO: approve this. Check whether clear enough what this section does beyond the last paper.
\begin{theorem}[{\cite[Theorem A.10]{CDFS1}}] \label{t:BspinEqs}
    Let $\cI_n$ be the ideal cutting out the orthogonal Grassmannian $\OGr(n,2n+1) \subset \PP^{2^n-1}$.
    Then $\cP_n$ generates $\cI_n$, and moreover forms a basis for the degree two graded piece of $\cI_n$.
    % %The set of equations 
    % \begin{align*}
    % \cI_n &= \ideal{p^M_{N,L}}{M,N,L \text{ pairwise disjoint} \, , \, 2|M| \in \cM_{N,L}} \subseteq \CC[x_I : I \subseteq [n]] \\
    % \text{ where } p^M_{N,L} &= \sum_{\substack{(\cubeVertex{K},\cubeVertex{\overline{K}}) \\\text{ antipodes in } Q_{N,L}}} s(K,M,N,L)x_{K}x_{{\overline{K}}} \, , \quad s(K,M,N,L) \in \{+1,-1\} \, .
    % \end{align*}
    % Moreover, this generating set forms a basis for the degree two graded piece of $\cI_n$.
\end{theorem}

This proves the first part of \Cref{thm:B}.
The second part follows from the following proposition.

\begin{proposition} \label{prop:tropicalisation}
    The tropicalisation of $\cP_n$ is equal to the peerless antipode equations $\cG_n$.
    %The tropicalisation of the basis $\{p^M_{N,L}\}$ in \cref{t:BspinEqs} is equal to the set of peerless antipode equations for $k$-cubes $Q$ with $k \geq 3$.
\end{proposition}

\begin{proof} 
    Both the peerless antipode equations and the equations in $\cP_n$ are supported on antipodes of subcubes of the ambient $n$-cube.
    There is a peerless antipode equation for every $k$-cube contained in the ambient cube with $k \geq 3$, and none for $k \leq 2$.
    If we can show the same for the equations in $\cP_n$, we are done.
    
    Recall that any $m$-cube can be expressed as $Q_{N,L}$ where $N,L$ are disjoint and $m = n-|N|+|L|$.
    Suppose $m \leq 2$. Then for any set $M \subseteq [n] \setminus (N \cup L)$, it follows that $0 \leq |M| \leq m \leq 2$.
    As $m-2, \dots, m+2 \notin \cM_{N,L}$, it follows that $2|M| \notin \cM_{N,L}$, and hence there is no equation in $\cP_n$ supported on the antipodes of $Q_{N,L}$.
    
    Now suppose $m \geq 3$.
    Verifying there exists some $M \subseteq [n] \setminus (N \cup L)$ with $2|M| \in \cM_{N,L}$ is a case check depending on $m \bmod 4$ and $n \bmod 2$.
    Note that
    \[
    4k + 2 \ceil{\frac{n}{2}} - |N| - |L| - \delta =
    \begin{cases}
        m + 4k - \delta & n \equiv 0 \bmod 2 \\
        m + 1 + 4k - \delta & n \equiv 1 \bmod 2
    \end{cases}
    \]
    Hence we get the following cases:
    % \begin{align*}
    %     m \equiv 0 \bmod 4 \, , \, n &\equiv 0 \bmod 2 & &\Rightarrow & 0 &= m - 4\left\lfloor\frac{m}{4}\right\rfloor \in \cM_{N,L} \\
    %     n &\equiv 1 \bmod 2 & &\Rightarrow & 0 &= m + 1 - 4\left\lfloor\frac{m}{4}\right\rfloor -1 \in \cM_{N,L} \\
    %     %
    %     m \equiv 1 \bmod 4 \, , \, n &\equiv 0 \bmod 2 & &\Rightarrow & 0 &= m - 4\left\lfloor\frac{m}{4}\right\rfloor -1 \in \cM_{N,L} \\
    %     n &\equiv 1 \bmod 2 & &\Rightarrow & 2 &= m + 1 - 4\left\lfloor\frac{m}{4}\right\rfloor \in \cM_{N,L} \\
    %     %
    %     m \equiv 2 \bmod 4 \, , \, n &\equiv 0 \bmod 2 & &\Rightarrow & 2 &= m - 4\left\lfloor\frac{m}{4}\right\rfloor \in \cM_{N,L} \\
    %     n &\equiv 1 \bmod 2 & &\Rightarrow & 2 &= m + 1 - 4\left\lfloor\frac{m}{4}\right\rfloor -1 \in \cM_{N,L} \\
    %     %
    %     m \equiv 3 \bmod 4 \, , \, n &\equiv 0 \bmod 2 & &\Rightarrow & 2m &= m + 4\left\lceil\frac{m}{4}\right\rceil + 1 \in \cM_{N,L} \\
    %     n &\equiv 1 \bmod 2 & &\Rightarrow & 0 &= m + 1 - 4\left\lceil\frac{m}{4}\right\rceil \in \cM_{N,L}
    % \end{align*}
    \begin{align*}
        m \equiv 0 \bmod 4 \quad &\Rightarrow \quad 
        \begin{cases}
            0 = m - 4\left\lfloor\frac{m}{4}\right\rfloor \in \cM_{N,L} & n \equiv 0 \bmod 2 \\
            0 = m + 1 - 4\left\lfloor\frac{m}{4}\right\rfloor -1 \in \cM_{N,L} & n \equiv 1 \bmod 2
        \end{cases} \\
        m \equiv 1 \bmod 4 \quad &\Rightarrow  \quad
        \begin{cases}
            0 = m - 4\left\lfloor\frac{m}{4}\right\rfloor -1 \in \cM_{N,L} & n \equiv 0 \bmod 2 \\
            2 = m + 1 - 4\left\lfloor\frac{m}{4}\right\rfloor \in \cM_{N,L} & n \equiv 1 \bmod 2
        \end{cases} \\
        m \equiv 2 \bmod 4 \quad &\Rightarrow \quad
        \begin{cases}
            2 = m - 4\left\lfloor\frac{m}{4}\right\rfloor \in \cM_{N,L} & n \equiv 0 \bmod 2 \\
            2 = m + 1 - 4\left\lfloor\frac{m}{4}\right\rfloor -1 \in \cM_{N,L} & n \equiv 1 \bmod 2
        \end{cases} \\
        m \equiv 3 \bmod 4 \quad &\Rightarrow \quad 
        \begin{cases}
            2m = m + 4\left\lceil\frac{m}{4}\right\rceil + 1 \in \cM_{N,L} & n \equiv 0 \bmod 2 \\
            0 = m + 1 - 4\left\lceil\frac{m}{4}\right\rceil \in \cM_{N,L} & n \equiv 1 \bmod 2
        \end{cases}
    \end{align*}
    Note that in every case, the claimed value in $\cM_{N,L}$ avoids $\{m-2, \dots, m+2\}$ as $m \geq 3$.
    In particular, there exists a set $M$ disjoint from $N$ and $L$ such that $2|M| \in \cM_{N,L}$.
    Thus $p^M_{N,L}$ is an element of $\cP_n$ that is supported on the antipodes of $Q_{N,L}$.
    % It is easy to see that the tropicalisation of $p^M_{N,L}$ does not depend on $M$ and
    % \[
    % \trop( p^M_{N,L}) =  \bigoplus_{I\in Q_{N,L}} x_I \odot x_{I^{c,T}_\face}.
    % \]
    % Hence every peerless antipode equation is contained in the tropicalisation of the equations given in Theorem \ref{t:BspinEqs}.
    % Further, since the tropicalisation is independent of $M$, each equation $p^M_{N,L}$ tropicalises to the peerless antipode equation on the subcube $Q_{N,L}$.
\end{proof}

% Combining Theorem \ref{thm:strong+iff+no+peerless}, which states that a subset of the cube is a $\Delta$ matroid if and only if it has no peerless antipodes, with the fact that the basis of Grassmann embedding equations $\{p^M_{N,L}\}$ tropicalise to the peerless antipode equations shows that a $M$ is a matroid if and only if the $\nu_M$ is a solution to $\trop\{ p^M_{N,L}\}$. 

\begin{remark} \label{rem:surprising tropicalisation}
There are multiple equations in $\cP_n$ that tropicalise to the same peerless antipode equation in $\cG_n$.
In particular, there are many equations in the span of~$\cP_n$ with strictly smaller support that are lost in the tropicalisation process.
Often in tropical geometry, one instead works with the equations of minimal support to avoid this, and this is precisely the approach taken in \cite[Appendix A]{CDFS1}.
The authors find it somewhat remarkable that the peerless antipode equations still characterise strong $\Delta$-matroids, despite the tropicalisation process vastly reducing the number of equations from $\cP_n$ to~$\cG_n$.
\end{remark}

\begin{example}\label{ex:B5}
   % We emphasise that the type $B$ quadratic embedding equations do not form a basis for the Lichtenstein embedding equations, as the following example demonstrates.
    Consider the $n=5$ case: this corresponds to the orthogonal Grassmannian $\mathrm{OGr}(5,11)$.
    It follows from \cite[Example 4.9]{CDFS1} that $\cP_5$ consists of $66$ equations.
    Moreover, there are exactly 6 equations supported on antipodes of the full $5$-cube, i.e. equations of the form $p^M_{\emptyset,\emptyset}$ where $M \subseteq [5]$ such that 
    \[
    2|M| \in \cM_{\emptyset,\emptyset} = \set{4k + 6 - \delta}{k \in \ZZ \, , \, \delta \in \{0,1\}} \setminus \{3,4,5,6,7\} \, .
    \]
    The $6$ such subsets are the five singleton sets and $[5]$.

    On the other hand, there is exactly one peerless antipode equation for each subcube of the $5$-cube of dimension $3$ or more.
    Thus $\cG_5$ consists of $\binom{5}{5} + 2 \binom{5}{4} + 2^2\binom{5}{3} = 51$ equations.
    We see that $\cP_5$ has $15$ more equations than $\cG_5$:
    $6$ equations in $\cP_n$ tropicalise to the unique peerless antipode equation supported on the full $5$-cube, and for each of the ten $4$-cubes contained in the $5$-cube, two equations in $\cP_n$ tropicalise to the unique peerless antipode equation supported on that $4$-cube.
    \end{example}

For general $n$, the cardinality of $\cG_n$ equals the number of subcubes of the $n$-cube of dimension $3$ or greater, which is $\sum_{k=3}^n 2^{n-k} \binom{n}{k} = 3^n - 2^n - n 2^{n-1} - \binom{n}{2} 2^{n-2}= \Theta(3^n)$; see sequence A345954 in \cite{OEISA345954} for a relationship to ternary strings.
On the other hand, the cardinality of $\cP_n$ can be calculated using the analysis in \cite{CDFS1} to be $2^{n-1} (2^n +1) - \binom{2n+1}{n} = \Theta(4^n)$.
%4^n/2 +  2^n/2 - 4^{n+1}/(2\sqrt{2n+3})\cdot\prod{factors less than 1}
% cf. https://planetmath.org/asymptoticsofcentralbinomialcoefficient
The first few values are listed in the table below.

\begin{center}
\begin{tabular}{ c | c c c c c c c c c  }
$n $ &3& 4 & 5  & 6 & 7 & 8 & 9 & 10  \\ 
\hline
$|\cP_n|$ & 1 & 10 & 66 & 364 & 1821 & 8586 & 38950 & 172084 &   \\  
$|\cG_n|$ & 1 & 9 & 51 &233  &939  & 3489 & 12259 & 41385   \\
\end{tabular}
\end{center}

\printbibliography

\end{document}